\documentclass[11pt]{amsart}

\usepackage{enumerate,url,amssymb,  mathrsfs}
\usepackage{graphicx}
\newtheorem{theorem}{Theorem}[section]
\newtheorem{lemma}[theorem]{Lemma}
\newtheorem*{lemma*}{Lemma}
\newtheorem{proposition}[theorem]{Proposition}
\newtheorem{corollary}[theorem]{Corollary}
\newtheorem{conjecture}[theorem]{Conjecture}

\theoremstyle{definition}
\newtheorem{definition}[theorem]{Definition}

\theoremstyle{remark}
\newtheorem{remark}[theorem]{Remark}

\numberwithin{equation}{section}

\newcommand{\C}{\mathbb{C}}

\newcommand{\DD}{\mathbb{D}}

\newcommand{\R}{\mathbb{R}}

\newcommand{\Hr}{\mathbb{H}}

\DeclareMathOperator{\re}{Re}

\DeclareMathOperator{\beur}{{\mathbf S}}

\def\Xint#1{\mathchoice
{\XXint\displaystyle\textstyle{#1}}%
{\XXint\textstyle\scriptstyle{#1}}%
{\XXint\scriptstyle\scriptscriptstyle{#1}}%
{\XXint\scriptscriptstyle\scriptscriptstyle{#1}}%
\!\int}
\def\XXint#1#2#3{{\setbox0=\hbox{$#1{#2#3}{\int}$}
\vcenter{\hbox{$#2#3$}}\kern-.5\wd0}}

\newcommand{\real}{{\mathbb R}}

\newcommand{\radial}{{\mathscr A}}
\newcommand{\openset}{{\mathscr O}}

\begin{document}
\baselineskip6mm
\title[Burkholder Integrals and Quasiconformal Mappings]{Burkholder Integrals, Morrey's Problem \\ and Quasiconformal Mappings}

\author{Kari\; Astala, \; Tadeusz\; Iwaniec,  \\ Istv\'an \;Prause, \; Eero \;Saksman}

\address{Department of Mathematics and Statistics,
University of Helsinki, Finland}
\email{kari.astala@helsinki.fi}

\address{Department of Mathematics, Syracuse University, Syracuse,
NY 13244, USA, and Department of Mathematics and Statistics,
University of Helsinki, Finland}
\email{tiwaniec@syr.edu}

\address{Department of Mathematics and Statistics,
University of Helsinki, Finland}
\email{istvan.prause@helsinki.fi}

\address{Department of Mathematics and Statistics,
University of Helsinki, Finland}
\email{eero.saksman@helsinki.fi}

\date{December 2, 2010}

\maketitle
\begin{abstract} Inspired by Morrey's Problem (on rank-one convex functionals) and the Burkholder integrals (of his martingale theory)  we find that the Burkholder functionals $\text{B}_p$, $p \geqslant 2$,  are quasiconcave, when tested on  deformations of identity $f\in Id + {\mathscr C}^\infty_\circ (\Omega )$ 
with $ \text{B}_p\,(Df(x)) \geqslant 0$ pointwise, or  equivalently, 
deformations
 such that $|D f |^2 \leqslant \frac{p}{p-2}J_f $.
 In particular, this holds in  explicit neighbourhoods of the identity map.
 Among the many immediate consequences, this gives
the strongest possible $\,\mathscr L^p$- estimates for the gradient of a principal solution to the Beltrami equation $\, f_{\bar{z}} = \mu(z) f_z\,$,  for any $p$  in the critical interval $2 \leqslant p \leqslant 1+1/\|\mu_f\|_\infty$.
Examples of local maxima lacking symmetry manifest the intricate nature of the problem.
\end{abstract}

\section{Introduction}

A continuous function $\textbf{E} : \mathbb R^{n\times n} \rightarrow \mathbb R\,$  is said to be  
${quasiconvex}$ if for every $\,f \in A + {\mathscr C}^\infty_\circ (\Omega, \mathbb R^n)$ we have
\begin{equation}\label{columbo}
\mathscr E[f] \;:=\; \int_ \Omega \mathbf E(Df)\, \textrm{d}x \geqslant  \int_ \Omega \mathbf E(A)\, \textrm{d}x\; =   \mathbf E(A) |\Omega|,
\end{equation}
where $\,A\,$ stands for an arbitrary  linear mapping (or its matrix) and $\Omega \subset  \mathbb R^n$ is any bounded domain. In other words, one requires that compactly supported  perturbations of linear maps do not
decrease the value of the integral. This notion is of fundamental importance in the
calculus of variations as it is known
to characterize lower semicontinuous integrals \cite{M}. A weaker notion is that of $rank$-$one$ $convexity$, which requires just that  $t\mapsto \textbf{E} (A+tX)  $ is convex for any fixed matrix $A$ and  for any rank one matrix $X.$
Rank-one convexity of an integrand  is a local condition and thus much easier to verify  than quasiconvexity. That quasiconvexity implies rank-one convexity was known  after Morrey's fundamental work in 1950's, but one had to wait until \v{S}ver\'{a}k's paper \cite{Sv} to find out that the converse is not true. 

However, \v{S}ver\'{a}k's example works only in dimensions $n\geqslant 3$, \cite{PS}. This leaves the possibility for different outcome in dimension 2, see \cite{FS}, \cite{Mu2} for evidence in this direction. Morrey himself was not quite definite in which direction he thought things should be true, see \cite{M}, \cite{M2}, and \cite[Sect. 9]{Ball3}. 
We reveal our own thoughts on the matter by recalling the following conjecture  in the spirit of Morrey:
\begin{conjecture}\label{Morrey's}
Continuous rank-one convex functions  $\textbf{E} : \mathbb R^{2\times 2} \rightarrow \mathbb R$ are quasiconvex.
\end{conjecture}

One says that
$\textbf{E}$ is rank-one concave (resp. quasiconcave) if $-\textbf{E}$ is  rank-one convex (resp. quasiconvex), and null-Lagrangian if both quasiconvex and quasiconcave. In the sequel we will rather discuss concavity, as this turns out to be  natural for our methods.
The most famous (and, arguably, the most important) rank-one concave function in two dimension is the Burkholder functional from \cite{Bu1},  defined for any $2\times 2$ matrix $A$ 
by
\begin{equation}
\label{burk}
 \textbf{B}_p\,(A) = \,\Big(\; \frac{p}{2}\det A  \; +\; (1-\frac{p}{2})\, \big{|}A\big{|}^2 \; \Big)\cdot |A|^{p-2},\qquad  p\geqslant {2}.
\end{equation}
Above, we have chosen the normalization $ \textbf{B}_p\,(Id\,) = 1$ with the identity matrix and the absolute value notation is reserved for the operator norm.
It is well known that among the wealth of important results the quasiconcavity of the Burkholder functional would  imply e.g.  the famous Iwaniec conjecture on the $p$-norm of the Beurling-Ahlfors operator.

The purpose of the present paper is to validate Conjecture \ref{Morrey's} in an important special case, namely for the above Burkholder functional, in case of  non-negative integrands, and for perturbations of the identity map:
\begin{theorem}\label{MainTh3} Let $\Omega\subset\real^2$ be a bounded domain and denote by $Id:\Omega\to\real^2$ the identity map. Assume that $f \in Id + {\mathscr C}^\infty_\circ (\Omega )$
satisfies  $\textbf{B}_p\,\bigl(Df(x)\bigr)\geqslant 0$  pointwise in $\Omega$. Then
$$
\int_\Omega \textbf{B}_p\,(Df)\, dx\, \leqslant  \int _\Omega \textbf{B}_p\,(Id\,)\, dx =|\Omega |, \qquad  p\geqslant {2},
$$
or, written explicitly
\begin{equation}\label{uusi22}
\int_\Omega  \,\Big(\; \frac{p}{2}\; J(z,f)   + ( 1 - \frac{p}{2})\; |Df|^2 \; \Big)\cdot |Df|^{p-2} \,  \leqslant \;
 |\Omega |.
\end{equation}
\end{theorem}
\medskip

Our proof of the above result  is based on holomorphic deformations and quasiconformal methods, and we next explain some of the relations between the Burkholder integrals and these maps. The reader is referred  to Section \ref{se:proof} for the needed notation. 
One observes that the condition $\textbf{B}_p\,(Df)\geqslant 0$ is  equivalent to
$|D f |^2 \leqslant \frac{p}{p-2}J_f $, which actually amounts to quasiconformality of $f$. 
In this setting our result reads as follows:

\begin{theorem}\label{MainTh2}
Let $\,f : \,\Omega{\longrightarrow}\,\Omega\,$ be a $\,K$-quasiconformal map of a bounded open set $\,\Omega \subset \mathbb C\,$  onto itself, extending continuously up to the boundary, where it coincides with the identity map $\,Id(z) \equiv z$.  Then
$$
\int_\Omega \textbf{B}_p\,(Df)\, dx\, \leqslant  \int _\Omega \textbf{B}_p\,(Id\,)\, dx \,=\, |\Omega |,
\quad\,  \textnormal{\textit{for all }} \;\, 2 \leqslant p \leqslant \frac{2K}{K-1}\,.
$$ 
Further, the equality occurs  for a class of (expanding) piecewise radial mappings discussed in Section \ref{se:radial}.
\end{theorem}
\noindent 
This result says, roughly, that the Burkholder functional is quasiconcave within  quasiconformal perturbations of the identity.
It is quite interesting that indeed  there is an equality in the above theorem for a large class of radial-like maps. When smooth and $p< 2K/(K-1)$ these are all local maxima for the functional, see Corollary \ref{locmaxima} for details. In particular, the identity map is a local maximum of all the Burkholder functionals in \eqref{burk}.

From the point of view of the theory of nonlinear hyperelasticity of
John Ball  \cite{Ball1,Ball2} and his collaborators \cite{An,C}, for homogeneous materials the elastic deformations  $\,f\colon  \Omega \rightarrow \mathbb R^n\;$ 
are  minimizers of  a given  energy integral
\begin{equation}
\mathscr E[f] \;=\; \int_ \Omega \mathbf E(Df)\, \textrm{d}x  \; <\infty 
\end{equation}
where the so-called \textit{stored energy function} $\, \mathbf E\colon  \mathbb R^{n\times n}\rightarrow \mathbb R\,$ carries the mechanical properties of the elastic material in $\,\Omega\,$. 
By virtue of the principle of non-interpenetration of matter the minimizers ought to be injective. It is from these perspectives that our energy-estimates, although limited to (quasiconformal) homeomorphisms, are   certainly not short of applications.

 Among the strong  consequences of the theorem,  one obtains (with the same assumptions as in Theorem \ref{MainTh2}) that
\begin{equation}\label{Lp}
\;\;\frac{1}{|\Omega|}\int_{\Omega}  \big | Df(z)\big|^{\,p }\;\textnormal d z
\;\leqslant \; \frac{2K}{ 2K \;-\; p\,(K-1)}\;,\quad for\;\; 2 \leqslant p < \frac{2K}{K-1}\,
\end{equation}\\
 with equality  for piecewise power mappings, such as $\, f(z) =  |z|^{1-1/K} z \,$ in the unit disk, see Corollary \ref{best} below. 
The  ${\mathscr W}^{1,p}$-regularity of $K-$quasiconformal mappings, for $p < 2K/(K-1)$, was established  by the first author in \cite{As2}, as a corollary of his area distortion theorem. However, there  the  bounds for integrals such as in (\ref{Lp}) were described  in terms of unspecified constants depending on the distortion $K$. Here we have obtained the sharp explicit bound for the ${\mathscr L}^p$-integrals of the derivatives of $K$-quasiconformal mappings.

Similarly, for any $K$-quasiregular mapping $f \in {\mathscr W}^{1,2}_{loc}(\Omega)$, injective or not, we can improve the local ${\mathscr W}^{1,p}$-regularity to weighted integral bounds at the borderline exponent $p=2K/(K-1)$,
\begin{equation}
\label{ }
\left( \frac{1}{K(x)}\; -\;\frac{1}{K} \right)\;\big| D\!f(x)\big| ^{\frac{2K}{K-1}} \in \,\mathscr L^1_{loc}(\Omega).
\end{equation}
We refer  to Section \ref{se:proof} for a more thorough discussion and Section \ref{se:radial} for elaborate examples of extremal mappings.

We next describe shortly the ideas behind the proofs of our main results which  are given in Section \ref{se:proof}. In fact we will prove a slightly generalized form of  Theorems  \ref{MainTh3} -- \ref{MainTh2}, where we relax the identity boundary conditions to asymptotic normalization at infinity.
This is done in Theorem \ref{MainTh} below, where we will interpolate between the natural end-point cases $p=2$ and $p=\infty$.
The holomorphic interpolation method used is inspired by the variational principle of thermodynamical formalism and the underlying analytic dependence coming from holomorphic motions. The latter tools already figured prominently in the proof of the area distortion theorem \cite{As2} by the first author.  

Here these are developed to a key ingredient of our argument,
a new variant of the celebrated  Riesz-Thorin interpolation theorem. We believe that the usefulness of this result may go beyond our interests here.  
In order to describe this result,
let $\,(\Omega,\,\sigma)\,$ be a measure space and let $\mathscr M(\Omega,\,\sigma) \,$ denote the class of complex-valued $\,\sigma$-measurable functions on $\Omega$. The Lebesgue spaces $\,\mathscr L^p (\Omega,\,\sigma)\,$ are  (quasi-)normed by
$$
  \| \Phi\|_p  = \left(\int_\Omega |\Phi(z)|^p \;\textrm{d}\sigma(z)\right)^{\frac{1}{p}}\,,\;\;\;0 < p<\infty\;, \;\;\; \textrm{and}\;\;\|\Phi\|_\infty \;= \underset{z\in\,\Omega}{\textrm{ess\,sup}} \,\;|\Phi(z)|
$$
Let $U\subset\C$ be a domain. We shall consider  analytic families $f_\lambda$ of measurable functions in $\Omega ,$  i.e.  jointly measurable functions $(x,\lambda)\mapsto f_\lambda (x)$ defined on
$\Omega\times U$ such that for each fixed $x\in \Omega$ the map $\lambda \to f(x,\lambda)$ is analytic in
$U$. The family is said to be {\it non-vanishing} if there exists a set $E\subset\Omega$ of $\sigma$-measure zero such 
\begin{equation}\label{vanishing}
f(x,\lambda)\not=0\qquad {\rm for \; all}\;\; x\in \Omega\setminus E\;\;{\rm and}\;\; \lambda\in U.
\end{equation}
We state our interpolation result first in the setting of the right half plane,
$U =  \mathbb H_+ := \{ \lambda: \; \re\lambda \,>\, 0\} $, in order to facilitate comparison with the Riesz-Thorin theorem:
\begin{lemma}[Interpolation Lemma]\label{interpolation2}
Let $\;0 \,<\, p_0, \, p_1 \leqslant \infty$ and
let   $\;\{\Phi_{\lambda}\,;\; \lambda \in  \mathbb H_+\}\, \subset\,\mathscr M(\Omega, \sigma)\,$ be an analytic and non-vanishing family,  with complex parameter $\lambda$ in the right half plane.
Assume further that for some $a\geqslant 0$,
$$
M_1\;:=\; \|\,\Phi\,_1\|_{p_1} \;\;<\;\infty  \;\; \textrm{and }\; M_0\;:=\; \underset{\lambda \in  \mathbb H_+}{\sup}
\,e^{- a\re \lambda} \|\,\Phi_\lambda\|_{p_0} \;<\infty. 
$$
Then,  letting  $\; M_\theta \;:=\; \big{\|}\Phi_\theta\big{\|}_{p_\theta}\;$ with \quad $\displaystyle
 \frac{1}{p_\theta}=(1-\theta )\cdot \frac{1}{p_0} + \theta\cdot \frac{1}{p_1},$\\
we have  for every $0<\theta< 1\,$,
\begin{equation} \label{logconcav}
M_\theta \leqslant  \;M_0^{1-\theta}\cdot M_1^{\theta}\;<\,\infty \;
\end{equation}
\end{lemma}

\begin{remark} Compared to Riesz-Thorin, our result needs the bound for the other end-point exponent only at one single point, when $\lambda = 1$ !  However,  without  the non-vanishing condition the  conclusion of the interpolation lemma breaks down drastically. A simple example (where $a=0$) is obtained by taking $p_0=1$, $p_1=\infty$,  and considering the
family $f(x,\lambda )=\left(\frac{1-\lambda}{1+\lambda}\right)g(x)$ for $\re\lambda >0$ and $x\in\Omega ,$
for a function  $g\in \mathscr L^1 (\Omega,\,\sigma)\setminus\left(\bigcup_{p>1} \mathscr L^p (\Omega,\,\sigma)\right).$
\end{remark}
In applications one often has rotational symmetry, thus requiring a unit disk version of the interpolation. After a  M\"obius transform in the parameter plane the interpolation lemma runs as follows (observe that we have interchanged the  roles of the indices $p_0$ and $p_1$ only for aesthetic reasons): 

\begin{lemma}[Interpolation Lemma for the disk]\label{interpolation}
Let $0<p_0 ,p_1 \leqslant \infty$ and
$\;\{\Phi_{\lambda}\,;\; |\lambda| \,<\, 1\}\, \subset\,\mathscr M(\Omega, \sigma)\,$ be an analytic and non-vanishing family  with complex parameter $\lambda$ in the unit disc.
Suppose
$$
M_0\,:=\, \|\,\Phi\,_0\|_{p_0} \;<\infty,\;\;\;  M_1\,:=\, \underset{|\lambda| <1}{\sup} \|\,\Phi_\lambda\|_{p_1} \;<\infty \;\;\; \textrm{and }\; M_r \,:=\, \underset{|\lambda|=r}{\sup}\,\big{\|}\Phi_\lambda\big{\|}_{p_r},
$$
where
\[ 
 \frac{1}{p_r}=\frac{1-r}{1+r}\cdot \frac{1}{p_0} + \frac{2r}{1+r}\cdot \frac{1}{p_1}\\
\]
\vskip8pt

\noindent Then, for every $\, 0\leqslant r < 1\,$,  we have
\begin{equation}
M_r \leqslant  \;M_0^{\frac{1-r}{1+r}}\cdot M_1^{\frac{2\,r}{1+r}}\;<\,\infty \;
\end{equation}
\end{lemma}

As might be expected, passing to the limit in (\ref{uusi22}) as $\,p\to 2\,$  or as $\,p\to\infty\,$ will yield interesting sharp inequalities.  The first mentioned limit leads to

\begin{corollary}\label{LlogL Th}  Given a bounded domain  $\Omega\subset\R^2$ and a homeomorphism  $\,f\, : \,\Omega \overset{\textnormal{\tiny{onto}}}{\longrightarrow}\,\Omega\,$ such that
  $$f(z) -  z \in {\mathscr W}^{1,2}_0(\Omega),$$ 
  we  then have
\begin{equation}\label{LlogL1} \int _\Omega \,\left(1\, + \,\log |Df(z)|^2 \,\right )\, J(z,f)\,\textnormal d z  \;\;\leqslant \;\;\int _\Omega  |Df(z)|^2  \,\textnormal d z
\end{equation}
Equality occurs for the identity map, as well as for a number of piece-wise radial mappings discussed in Section \ref{se:radial}.
\end{corollary}
\noindent Reflecting back on Conjecture \ref{Morrey's}, note that the functional $${ \mathscr  F}(A) = \,\left(1\, + \,\log |A|^2 \,\right )\, \det(A)\, - \, |A|^2\, $$ is 
 rank-one concave. However, with growth stronger than quadratic it is not polyconcave \cite{Ball1}, i.e.~cannot be written as a concave function of the minors of $A$. According to   Corollary \ref{LlogL Th} this functional is nevertheless quasiconcave with respect to  homeomorphic perturbations of the identity. Going to the inverse maps yields a quasiconcavity result for the functional
\begin{equation}\label{functional}
{\mathscr H} (A):=\frac{1}{2} \frac{|A|^2}{{\rm det}\, A}+\log \big( {\rm det}\, A\big)-\log |A|, \qquad \;\; {\rm det }\, A >0.
\end{equation}
 This can be interpreted as a sharp integrability of $\,\log J(z,f)\,$ for planar maps of finite distortion, see Corollary \ref{loginv}  below.  
 
It is  of course classical  \cite{Mu,GI} that the nonlinear differential expression
$J(z,f) \,\log \,|Df(z)| ^2$ belongs to $\;\mathscr L^1_{\textrm{loc}} (\Omega)$
  for every $\; f \in {\mathscr W}^{1,2}_{\textnormal{loc}} (\Omega)\,$ whose
Jacobian determinant $\, J(z,f) = \textnormal{det}\,Df(z)\,$ is nonnegative.
 The novelty in (\ref{LlogL1}) lies in the best constant $\,C= 1$ in the right hand side, and the proof of  $\,\mathscr L\textnormal{log}\mathscr L$-integrability of the Jacobian is new. 
 
In turn, the limit $\,p\to \infty\,$ yields the following sharp inequality.
\begin{corollary}\label{expint}
Denote by ${\mathbf S}$ the Beurling-Ahlfors operator (defined in (\ref{56})) and assume that $\mu$ is a measurable function with $|\mu (z)|\leqslant \chi_{\mathbb D}(z)$ for every $z \in \mathbb{C}$. Then
\begin{equation}\label{expint1}\int_{\mathbb D} \;\big(\, 1 - |\mu (z)|\big)\; 
e^{|\mu(z)|}\; \bigl| \exp({\mathbf{S}\mu(z)}) \bigr| \;dz \;\leqslant \; \pi .
\end{equation}
Equality occurs  for an extensive class of piece-wise radial mappings discussed in Section \ref{LlogLeq}.
\end{corollary}
\noindent Prior to the above result, it was known that the area distortion results \cite{As2} yield the exponential integrability of $\re {\mathbf S}\mu$  under the strict inequality
$\|\mu \|_\infty <1$, see \cite[p. 387]{AIMb}.  

The proofs of the Corollaries are found in Section \ref{se:limitcases}.
Section \ref{se:radial} contains further observations on the Burkholder functionals. For example,  their local maxima are  discussed and  new conjectures are posed also in the higher dimensional setup. We finally mention that an extension of Burkholder functionals to all real values of the exponent $p$ is given in Section \ref{se:proof}.

\section{Proof of the Interpolation Lemma}\label{se:interpolation}

Before embarking into the proof, let us remark that often analytic families of functions are defined
by considering analytic functions having values in the Banach space $\,\mathscr L^p (\Omega,\,\sigma)\,$
for $p\geqslant 1.$ In this case it is well-known (e.g. \cite[Thm. 3.31]{R}) that one may define analyticity of the family by several equivalent conditions, e.g. by testing elements from the dual. This notion agrees with the definition given in the introduction, see \cite[Lemma 5.7.1]{AIMb}.
\smallskip

\begin{proof}[Proof of Lemma \ref{interpolation2}] 
We may, and do, assume that   $\,M_0 = 1\,$ and that $a=0$; the case $a>0$ reduces to this by simply considering the analytic family $e^{-a\lambda}\Phi_\lambda (x)$. Similarly by taking restrictions we may assume
$\sigma(\Omega) < \infty$.

We first consider the case $0<p_0, p_1<\infty,$ and   establish the result  in the situation  where for a fixed $A\in (1,\infty )$ there  is the uniform bound
\begin{equation}\label{uniform}
\frac{1}{A}\leqslant \left| \Phi_\lambda (x) \right| \leqslant A\qquad {\rm for \; all}\quad \lambda\in \Hr_+\;\;{\rm and}\;\;
x\in\Omega.
\end{equation}
This is to ensure that all of our integrals and computations below are meaningful.
At the end of the proof we  get rid of this extra assumption.

Let $\theta \in (0,1)$ be given as in the statement of the lemma. First, we will find the support line to the convex function $\frac1p \mapsto  \log \| \Phi_\theta \|_p$ at $\frac{1}{p_\theta}$. We are looking for a function $u_p(\theta)$ with the following properties,
\begin{equation} 
\label{eq:supportline}
u_p(\theta) = \frac{1}{p} \, I + u_\infty(\theta) \leqslant  \log \| \Phi_\theta \|_p  \quad \text{and} \quad u_{p_\theta}(\theta)= \log \| \Phi_\theta \|_{p_\theta},
\end{equation}
where $I$ and $u_\infty(\theta)$ are independent of $p$. Using the concavity of the logarithm function we can write down these terms explicitly. Indeed, by concavity, for any probability density $\wp(x)$ on $\Omega$ and for any exponent $0< p < \infty$,
$$
 \frac{1}{p}\int_\Omega\, \wp(x) \log \left(\frac{|\Phi_\theta(x)|^{p}}{\wp(z)}\right)\;\textrm{d}\sigma \leqslant \log \| \Phi_\theta \|_{p},
$$
where equality holds for $p=p_\theta$ with the following choice of density
\begin{equation}
\label{dens}
 \wp(x) \;: = \; \frac{\big|\Phi_\theta(x)\big|^{p_\theta}}{\int_\Omega \big|
\Phi_\theta (y) \big|^{p_\theta}\;\textrm{d}\sigma (y)}, \quad
\int_\Omega \;\wp(x)\,\textrm{d}\sigma (x)\;= 1.
\end{equation}
It is useful to note that because of our assumptions \eqref{uniform}, the $\wp$ is uniformly bounded from above and below.
With this in mind we find the coefficients in \eqref{eq:supportline}, by using the fixed density  \eqref{dens} and by writing 
\[ I:= \int_\Omega\, \wp(x) \log \left(\frac{1}{\wp(z)}\right)\,\textrm{d}\sigma
\quad \text{and} \quad u_\infty(\theta):=\int_\Omega\, \wp(x) \log |\Phi_\theta(x)|\;\textrm{d}\sigma.
\]

The key idea in this representation is that  we may embed the line $u_p(\theta)$ in a harmonic family of lines parametrized by 
$\lambda \in \mathbb{H_+}$,  

\[ u_p(\lambda):= \frac{1}{p} \, I + u_\infty(\lambda) \;=\; \frac{1}{p}\int_\Omega\, \wp(x) \log \left(\frac{|\Phi_\lambda(x)|^{p}}{\wp(z)}\right)\;\textrm{d}\sigma.
\]
It is important  to notice  that we kept the slope $I$ fixed and because of the non-vanishing assumption the constant term $u_\infty(\lambda)$  becomes a harmonic function of $\lambda$. 
Again, in view of Jensen's inequality
we have the envelope property, the analogue of \eqref{eq:supportline} for all $\lambda \in \mathbb{H_+}$ and $0<p \leqslant \infty$,
\begin{equation}
\label{eq:envelope}
 u_p(\lambda) \leqslant \log \| \Phi_\lambda \|_p \quad \text{and} \quad u_{p_\theta}(\theta)= \log \| \Phi_\theta \|_{p_\theta}.
\end{equation}

By our assumptions, for $p= p_0$ we thus have  $u_{p_0}(\lambda) \leqslant \log \|\Phi_\lambda\|_{p_0} \leqslant 0$ for all $\lambda \in \Hr_+$. Here  Harnack's inequality for nonpositive harmonic functions in  $\mathbb H_+$ takes a particularly simple form when  restricted to 
the interval $\theta \in (0,1)$: 
\begin{equation}
\label{eq:harnack}
u_{p_0}(\theta)\leqslant \theta\, u_{p_0}(1)\qquad {\rm for}\;\; \theta \in (0,1).
\end{equation}
Finally combining  the estimates \eqref{eq:envelope} and \eqref{eq:harnack} yields
 \begin{eqnarray*}\label{kaava}
 \log \|\Phi_\theta\|_{p_{\theta}}  &=& u_{p_\theta}(\theta) = u_{p_0}(\theta)  + 
 \left( \frac{1}{p_\theta}-\frac{1}{p_0} \right) I  \\ \\
&\leqslant& \theta \, u_{p_0}(1) + \theta \left( \frac{1}{p_1}-\frac{1}{p_0} \right) I\\  \\
& = &  \theta \, u_{p_1}(1) 
\leqslant \theta \, \log M_1,
\end{eqnarray*}
which is exactly what we aimed to prove.

The argument can easily be adapted to accommodate the cases when $p_0=\infty$  or $p_1=\infty$. We will instead use a limiting argument. First, normalize to  $\sigma (\Omega )=1$, then $\| \cdot \|_{p}$ increases with $p$ and one has $\| f\|_\infty=\lim_{p\to\infty}\| f\|_p .$ Hence, as \eqref{uniform} holds, we
obtain the desired result by  approximating the possibly infinite exponents by finite ones.
  
Let us finally dispense with the extra assumption (\ref{uniform}). Since the removal of a null set from $\Omega$ is allowed, we may assume that the non-vanishing condition holds
for every $x\in\Omega ,$ i.e.  one may take  $E=\emptyset$ in (\ref{vanishing}). 
Choose first a family $\varphi_n(\lambda)$ of M\"obius transformations such that
$$\varphi_n(1) = 1, \qquad \lim_{n\to \infty} \varphi_n (\lambda )=\lambda, \;\; \lambda\in \Hr_+, \quad \; \mbox{ and }\;\; \,\overline{\varphi_n (\Hr_+ )} \subset \Hr_+, \;\;  n \in \mathbb N,
$$
and let  for any positive integer $k$, 
\begin{equation}\label{goodset}
\Omega_{n,k}:=\{ x\in\Omega\; :\; |\Phi_\lambda (x)|\in [1/k, k]\;\; {\rm for \; all }\;\; \lambda\in \overline{\varphi_n (\Hr_+ )} \; \} .
\end{equation}
The measurable sets $\Omega_{n, k}$  fill the space,
$\bigcup_{k=1}^\infty \Omega_{n ,k}=\Omega$. Moreover, for each fixed  integer $k\geqslant 1$ the non-vanishing analytic family
$$
(x,\lambda )\mapsto \ \Phi_{{\varphi_n} (\lambda)}(x), \qquad x \in \Omega_{n,k},  \;\; \lambda\in \Hr_+, 
$$
satisfies the uniform bound (\ref{uniform}), and thus we may interpolate it. Letting  $k\to\infty$ gives $\|\Phi_{\varphi_n (\theta)}\|_{p_\theta}\leqslant M_1^\theta$, and  the claim \eqref{logconcav} follows by Fatou's lemma taking a second limit $n \to \infty$.  
\end{proof}

We remark that the Harnack inequality used above can be deduced from the standard Harnack inequality in the unit disc $\,u(w) \leqslant \;\frac{1-|w|}{1+|w|}\, u(0)$ by a change of variables $\lambda=(1-w)/(1+w). $ The very same change of variables allows one to deduce Lemma \ref{interpolation} for the values $\lambda\in (0,1)$ as a consequence of Lemma \ref{interpolation2}, and rest follows from rotational symmetry.

\section{Proofs of the Main Theorems}\label{se:proof}

We start with some preliminaries. Our goal is to  apply the Interpolation Lemma \ref{interpolation} in estimating  the variational integrals   such as \eqref{uusi22}, and therefore we look for  analytic and  nonvanishing families of  gradients of mappings. In view of the {\it Lambda-lemma} \cite{MSS} this takes us  to the notion of  quasiconformal mappings.  
By definition, in any dimension $n \geqslant 2$ these are homeomorphisms $f:\Omega \to \Omega'$  in the Sobolev class $\mathscr W^{1,n}_{loc}(\Omega)$ for which 
the differential matrix $D\!f(x) \in \mathbb R^{n\times n}$  and its determinant are coupled in the \textit{distortion inequality},
\begin{equation}\label{distortion}
 |D\!f(x)|^n  \leqslant K(x)\, \det D\!f(x)\;,\quad \textrm{where}\;\;\;|D\!f(x)|  = \max_{|\xi| =1} \; |D\!f(x) \xi|,
\end{equation}
for some bounded function $K(x)$. The smallest  $\,  K(x) \geqslant 1 \, $ for which (\ref{distortion}) holds almost everywhere is referred to as the  \textit{distortion function} of the mapping $\,f$. We call $f$ $K$-quasiconformal if  $K(f) := \| K(x)\|_{\infty} \leqslant K$. 

In dimension $\,n=2\,$   it is useful to employ complex notation by introducing the  Cauchy-Riemann operators 
$$
\partial f = f_z= \frac{\partial f}{\partial z} = \frac{1}{2} \left(\frac{\partial f}{\partial x} - i \frac{\partial f}{\partial y}\right )\; \;\;\;\textrm {and}\;\;\;\;  \bar{\partial} f = f _{\bar z} = \frac{\partial f}{\partial \bar z} = \frac{1}{2} \left(\frac{\partial f}{\partial x} + i \frac{\partial f}{\partial y}\right)   $$
Writing
$$
  |D\!f(z)|  \,= |f_z|\;+\;|f_{\bar{z}}|\,\quad \;\;\;\textrm{and}\quad\;\; \det\,D\!f(z) \, = \,J(z,f) \,= \,|f_z|^2\,-\,|f_{\bar{z}}|^2,
$$
we see that for a planar Sobolev homeomorphism $f$ the $K$-quasiconformality simplifies to a linear
equation
\begin{equation}\label{Bel1}
\, f_{\bar{z}} = \mu(z) f_z
\end{equation}
called the {\it Beltrami equation}. Here the {\it dilatation function} $\mu$ is measurable and
satisfies
$\| \mu\|_\infty =:k=\frac{K-1}{K+1}<1$. 

The Beltrami equation  will then enable holomorphic deformations of the homeomorphism $f$ and of its gradient. Indeed,  under a proper normalization the solutions  to  \eqref{Bel1} and  their derivatives  depend holomorphically on the coefficient $\mu$, see \cite[p. 188]{AIMb}.  

For choosing the normalization,  recall that  Theorems \ref{MainTh3} - \ref{MainTh2}  consider identity boundary values, and thus mappings that extend conformally outside $\Omega$. We therefore look for solutions to   \eqref{Bel1} defined in the entire plane $\mathbb C$, with the dilatation $\mu$ vanishing outside the  domain $\Omega$.  
On the other hand,  the identity boundary values cannot be retained under general holomorphic deformations; one needs to content 
with the asymptotic normalization
\begin{equation}\label{principal}
\,f(z) = z + b_1 z^{-1} + b_2 z^{-2} + \cdots \;,\;\textnormal{for}\;\; |z| \to \infty 
\end{equation}
Global ${\mathscr W}^{1,2}_{loc}$-solutions  to \eqref{Bel1} with these asymptotics  are called  
{\it  principal solutions}. They exist and are unique for each coefficient $\mu$  supported in the bounded domain $\Omega$, and each of them is a homeomorphism. They can be found simply in the form of the Cauchy transform
 \begin{equation}
f(z) = z + \frac{1}{\pi} \int_\mathbb C \frac{\omega(\xi)\, \textnormal d \xi}{ z - \xi }\;,\;\;\;\textnormal{where} \;\; \omega = f_{\bar{z}} \in \mathscr L^2(\mathbb C)
\end{equation}
Substituting  $\,\omega:= f_{\bar z}$ into (\ref{Bel1}) yields a  singular integral equation for the unknown density function $\,\omega \in \mathscr L^2(\mathbb C)$,
 \begin{equation}\label{SBel}
\omega \;-\; \mu\, \mathbf S \omega \;=\;\mu\;\in \mathscr L^2(\mathbb C)
\end{equation}
Here the \textit{Beurling Transform}, a  Calder\'on-Zygmund type singular integral,
\begin{equation}\label{56}
 \mathbf S \omega \;= \;- \frac{1}{\pi}  \int_\mathbb C \frac{\omega(\xi)\, \textnormal d \xi}{ (z - \xi)^2 }\;= \; f_z\,-\, 1
\end{equation}
is an isometry in $\mathscr  L^2(\mathbb C)$, whence (\ref{SBel}) can be solved by the Neumann series. We refer to the well-known monographs \cite{Ah}, \cite{LV} and \cite{AIMb} for the basic properties and further details on quasiconformal mappings.

Regularity results in $\mathscr W^{1,p}_{loc}$ originated in \cite{Bo2}, \cite{Bo1}, where  $\mathbf S$ was considered on $\mathscr L^p$ for suitable exponents $p>2$. 
As a Calder\'on-Zygmund operator $ \mathbf S\; $ is bounded in  $\,\mathscr L^p$, but determining here the operator norm is a much harder question. The as yet unsolved conjecture \cite{Iw4} of Iwaniec asserts that
\begin{conjecture}\label{p-norm}
For all $\,1<p<\infty\,$ it holds
\begin{equation}\label{55}
\|\mathbf{S} \|_{\mathscr  L^p(\mathbb C)} = p^* - 1 := \;\left\{\begin{array}{ll} p-1\;,\quad\qquad\; \textnormal{if}\;\;\; 2\leqslant p < \infty\\
       \;\;1/(p-1) \;,\;\;\;\;\textnormal{if}\;\;\; 1<p\leqslant 2
       \end{array}\right.
\end{equation}
\end{conjecture}
\smallskip

\noindent As mentioned in the introduction, the full quasiconcavity of the Burkholder functional ${\bf B}_p$ would, among its many potential consequences, imply also  \eqref{55}.
This follows from another very useful inequality of Burkholder \cite{Bu1}. Namely,  with the positive constant $\,C_p = p\left(1 - \frac{1}{p}  \right)^{1-p}$ for $\, p \geqslant2\,$,  we have
\begin{equation*}
C_p \cdot \big(\,|f_z|^p \; - (p-1)^p |f_{\bar{z}}|^p \bigr)  \leqslant \, \big(\,|f_{z}|\,-\,(p-1)\,|f_{\bar{z}}| \big)\cdot \big(\,|f_z|\;+\;|f_{\bar{z}}|\big)^{p-1}\,\equiv  \, {\bf B}_p(Df)
\end{equation*}
Thus  Burkholder's functional can be viewed as a rank-one concave envelope of the 
$p$-norm functional of the left hand side. It is because of this connection why Morrey's Problem becomes relevant to Conjecture \ref{p-norm}.
For precise statements and further information on related topics we refer to \cite{Ba1,BL,Bu1,VN}. 

It is also appropriate to note that the origin of the Burkholder functional lies in Burkholder's groundbreaking work on  sharp estimates for martingales  \cite{Bu4} -  \cite{Bu1}. This work has been later on extended in various ways, including  applications to computing optimal or almost optimal estimates for norms of singular integrals, e.g. of the Beurling-Ahlfors operator. Also the Bellman function techniques (see e.g.~\cite{NTV2}) are closely related. We mention only
  \cite{BJ2}, \cite{BW}, 
 \cite{Bu1},  \cite{DV}, \cite{GMS},    \cite{NTV2},   \cite{NTV},  \cite{PV},    \cite{VN}      
and refer to the  recent survey \cite{Ba1} for a wealth of information and an extensive list of references.

We next recall  some standard facts on smooth approximation of principal solutions of the Beltrami equation. Especially, we have a special interest in the Jacobian and whether it is strictly positive. Schauder's regularity theory of elliptic PDEs with H\"{o}lder continuous coefficients becomes useful. It applies to general quasilinear Beltrami systems, see Theorem 15.0.7 in  \cite{AIMb}.
$$
 \frac{\partial f}{\partial\bar{z}}\; = \;\mu(z,f)  \frac{\partial f}{\partial z} \;+ \; \nu(z,f) \overline{\frac{\partial f}{\partial z}}\;,\quad\;\;\;|\mu(z,f)| + \;| \nu(z,f)| \leqslant k < 1
$$
where the coefficients $\,\mu\, ,\,\nu\, \in \mathscr C^\alpha( \mathbb C\times \mathbb C)\,,\,\; 0<\alpha<1\,$. Accordingly, every  solution $\,f\in\mathscr W^{1,2}_{\textnormal{loc}}(\mathbb C \times \mathbb C)$ is  $\,\mathscr C^{1,\,\alpha}_{\textnormal{loc}}$-regular. Of course, if $\,f\,$ is quasiconformal, then the inverse map $\, h(w)  = f^{-1}(w)\, $ is also $\,\mathscr C^{1,\,\alpha}_{\textnormal{loc}}$-regular  since it solves similar system
$$
 \frac{\partial h}{\partial\bar{w}} \;=\; -\nu(h,w) \, \frac{\partial h}{\partial w} \;- \; \mu(h,w)\, \overline{\frac{\partial h}{\partial w}}
$$
Thus, we have
\begin{lemma}[$\mathscr C^{1,\,\alpha}$-regularity]\label{aregu}
The principal solution of the Beltrami equation (\ref{Bel1}) in which $\,\mu\,\in \mathscr C^\alpha( \mathbb C)\,,\; 0<\alpha<1\,,\,$ is a
 $\,\mathscr C^{1,\,\alpha}( \mathbb C )\,$ - diffeomorphism. In particular, $\,|f_z|^2 \geqslant J(z,\,f) > 0\,$,  everywhere.
\end{lemma}
As might be expected, almost  everywhere convergence of the Beltrami coefficients yields $\,\mathscr W^{1,\,2}_{\textnormal{loc}}$- convergence of the principal solutions. The precise statement 
reads as follows:
\begin{lemma}[Smooth Approximation]
Suppose the Beltrami coefficients  $\,\mu_{_\ell}\, \in \mathscr C_\circ^\infty(\Omega)\,$  satisfy \, $\,|\mu_{_\ell} (z)|  \leqslant k < 1 $ , for all  $\,\ell = 1,2,... $,  and converge almost everywhere to $\,\mu\,$. Then the associated principal solutions $\, f^\ell\,:\, \mathbb C \rightarrow \mathbb C\,$ are $\,\mathscr C^\infty$-smooth diffeomorphisms converging in $\,\mathscr W^{1,\,2}_{\textnormal{loc}}(\mathbb C) \,$ to the principal solution of the limit equation $\,f_{\bar{z}} \,=\, \mu(z) f_z\,$.
\end{lemma}

Every measurable Beltrami coefficient satisfying $\,|\mu(z)| \,\leqslant \; k\, \chi _{_\Omega} (z)\;,\;\;\; 0\leqslant k  <1\,$,  can be approximated this way.

As the last of the preliminaries, in applying  the Interpolation Lemma \ref{interpolation} we will 
need  sharp $\mathscr{L}^2$-estimates of  gradients, valid for all complex deformations of a given mapping. For the principal  solutions in the unit disk $\mathbb D$, these result from the classical  Area Theorem (see e.g.\cite[p. 41]{AIMb}).
\begin{lemma}[Area Inequality] \label{lem:areaineq}
The area of the image of the unit disk under a principal solution in  $\mathbb D$ does not exceed $\,\pi$. It equals $\,\pi\,$ if and only if the solution is the identity map outside the disk.
\end{lemma}
\noindent Let us recall a proof emphasizing the null-Lagrangian property of the Jacobian determinant. On the circle we have the equality
 $\,f(z) =  g(z)\,$,  where $\,g \in \mathscr W^{1,2}(\mathbb D) \,$ is given by $\,g(z) =  z + \underset{n\geqslant 1}{\sum}\, b_n \bar{\;z\,}^{n}\;$, for $\,|z| \leqslant 1\,$. 
 This yields
$$
\int_\mathbb D \,J(z, f)\,\textrm{d}z = \int_\mathbb D \,J(z, g)\,\textrm{d}z = \int_\mathbb D \big(1 -|g_{\bar{z}}(z)|^2\,\big)\,\textrm{d}z  \leqslant \pi
$$
Equality occurs if and only if $\,g_{\bar{z}} \equiv 0 \,$, meaning that all the coefficients $b_n \,$ vanish.

Having disposed of these lemmas, we can now proceed to the proof of the main integral estimate, where
in the  complex notation the  Burkholder functional takes the form  \begin{equation}\label{Burkh1}
 \mathscr B_\Omega ^{\,p}\,[f] \,:=\int_{\Omega} \big(\;|f_{z}|\,-\,(p-1)\,|f_{\bar{z}}| \;\big)\cdot \big(\;|f_z|\;+\;|f_{\bar{z}}|\;\big)^{p-1}\, \textnormal d z\;,\;\;p \geqslant 2.
\end{equation}
We will actually  deduce (\ref{uusi22}) from a slightly more general result, where 
we  relax the identity boundary values  and allow  principal mappings:

\begin{theorem}[Sharp  $\,{\mathscr L}^p$-inequality]\label{MainTh}
Let $\,f :\,\mathbb C \rightarrow \mathbb C\,$ be the principal solution of a Beltrami equation;
\begin{equation}\label{6}
f_{\bar{z}}(z) \;= \mu(z)\; f_z(z) \;, \quad \quad |\mu(z)| \,\leqslant \; k\, \chi _{_\mathbb D} (z)\;,\;\;\; 0\leqslant k  <1,
\end{equation}
in particular,  conformal outside the unit disk $\mathbb D$. 

Then, for all exponents $\,2\leqslant p \leqslant 1 + 1/k$,  we have \\
\begin{equation}\label{Mainineq}
\int_{\mathbb D} \left(1\; -\; \frac{p\, |\mu (z)|}{1+\,|\mu (z)|}\right)\; \big(|f_{{z}}(z)|+ |f_{\bar{z}}(z)|\big)^{\,p }\;\textnormal d z
\;\;\leqslant \;\pi.
\end{equation}\\
Equality occurs for some fairly general  piecewise radial mappings discussed in Section \ref{se:radial}.
\end{theorem}
\noindent The above form of the main result gives a flexible and remarkably precise local description of  the  $\,\mathscr L^p$-properties  of derivatives of a quasiconformal map,  especially interesting in the borderline
situation $p=1+1/k$. Indeed, combined with the Stoilow factorization, the theorem gives for any  ${\mathscr W}^{1,2}_{loc}(\Omega)$-solution to (\ref{6}), injective or not, the estimate
\begin{equation} \bigl(k - \, |\mu (z)|\bigr)\; \big | Df(z)\big|^{\,1 + 1/k }\; \in \,\mathscr L^1_{loc}(\Omega).
\end{equation}
Thus  for all $K$-quasiregular mappings we obtain optimal weighted higher integrability bounds  at the borderline case $p =  2K/(K-1)$. For $p $ below the borderline, the ${\mathscr W}^{1,p}_{loc}$-regularity was established already in \cite{As2}.
The borderline integrability was previously covered \cite{AsN} only in the very special case $|\mu|=k \cdot \chi_E$, for $E \subset \mathbb{D}$, in Theorem \ref{MainTh}.
 
The proof of Theorem \ref{MainTh} applies the Interpolation lemma in conjunction with analytic families of quasiconformal maps. However, the choice of the specific analytic family  for our situation is quite non-trivial, in order to enable sharp estimates. In a sense the speed of the change with respect to the analytic parameter must be localized in a delicate manner, see (\ref{eq:tau}) below.

\begin{proof}[Proof of Theorem \ref{MainTh}]
Given a principal solution $f$ to \eqref{6}, with $\mu(z) \equiv 0$ for $|z|>1$, we are to prove the integral bounds (\ref{Mainineq}).
There is no loss of generality in assuming that  $\,\mu \in \mathscr C_\circ ^\infty(\mathbb D)\,$, for if not, we approximate $\,\mu\,$ with  $\,\mathscr C_\circ^\infty$-smooth Beltrami coefficients,
and, thanks to Fatou's lemma, there is no difficulty in passing to the limit in (\ref{Mainineq}). On the other hand, this reduction could be avoided by using the argument of  Remark \ref{rem:nonvanishing} below.

 With this assumption we fix an exponent $2 \leqslant p \leqslant 1+ \|\mu\|_\infty^{-1}$ and look for  holomorphic deformations of the given function $f$, via an analytic family of  Beltrami equations together with their principal solutions,
  \begin{equation}\label{defo}
   F^\lambda_{\bar z}  =  \mu_{_\lambda}(z)\, F^\lambda_z,\; \qquad   \mu_{_\lambda}(z) = \tau_\lambda(z) \cdot \frac{\mu(z)}{|\mu(z)|} 
\end{equation}
Here $\tau_\lambda(z)$ is an analytic function in $\lambda$ to be chosen later  with $|\tau_\lambda(z)| < 1$. 
We aim to explore Interpolation in the disk, Lemma \ref{interpolation}, 
by applying it to a suitable non-vanishing analytic family constructed from the derivatives of $F^\lambda(z)$. Hence the question is the right choice of $\tau_\lambda$.

We want $F^0(z) \equiv z$, thus $\tau_0(z) \equiv 0$, while for some value $\lambda = \lambda_\circ$
we need to have $\tau_{\lambda_\circ}(z) = |\mu(z)|$, so that  $ f = F^{\lambda_\circ} $. 
Comparing the exponents in  (\ref{Mainineq}) and  in Lemma \ref{interpolation} suggests that we choose
\begin{equation}\label{pee}
p = 1 + \frac{1}{\lambda_\circ}, \qquad p_0 = \infty, \qquad p_1 = 2.
\end{equation}
These conditions will  then be confronted with the need of  weighted $\mathscr{L}^2$-bounds consistent with  the inequality \eqref{Mainineq}.
 
 To make the long story short, we choose 
 \begin{equation} \label{eq:tau}
\, \mu_{_\lambda}(z) = \tau_\lambda(z) \cdot \frac{\mu(z)}{|\mu(z)|}, \;\; \;\textnormal{ where } \; 
\frac{\tau_\lambda(z)}{1+\tau_\lambda(z)} = p \cdot \frac{|\mu(z)|}{1+|\mu(z)|} \cdot \frac{\lambda}{1+\lambda}, 
\end{equation}
 or more explicitly,
$$
   \mu_{_\lambda}(z)\, =\, \frac{p\,\lambda\, \mu(z)}{(1+\lambda)\, (1\,+\; |\,\mu(z)|\,)\;  - \; p \,\lambda \;|\,\mu(z)|}
$$
 The complex parameter  $\lambda$ runs over the unit disk, $\,  |\lambda| < 1\,$. One may visualize 
 $\lambda \mapsto \tau_\lambda(z)$ as the conformal mapping  from the unit disk onto the horocycle
 $$ \Big\{ w \in \mathbb D :  \; 2 \, \re \left( \frac{w}{1 + w}\right)  <  p \cdot \frac{|\mu(z)|}{1+|\mu(z)|} \Big\}
 $$
 determined by the weight function in \eqref{Mainineq}.

From \eqref{eq:tau} one readily sees that $\,|\mu_{_\lambda}(z) |  \leqslant |\lambda|\, \chi _{_\mathbb D}(z) $, furthermore $\,\mu_{_\lambda} \in  \mathscr C^{\alpha}(\mathbb C)\,,\,0<\alpha \leqslant 1\,$  . Therefore the equation (\ref{defo}) admits a unique principal solution $\,F^\lambda \,:\mathbb C \rightarrow \mathbb C\,$, which is a $\,\mathscr C^{1,\,\alpha}$ - diffeomorphism. It depends analytically  \cite{AhBers} on the parameter $\,\lambda\,$, as seen by  developing \eqref{SBel}  in a Neumann series, and we have
 $$
       |F^\lambda_z|^2 \;\geqslant \;|F^\lambda_z |^2 \,-\, |F^\lambda_{\bar z} |^2 \; = \; J(z,F^\lambda) \; > \;0\;, \quad \textnormal { everywhere in }\; \mathbb C.
 $$
 Moreover, $F^0(z) = z$ with $F^{\lambda_\circ} = f$, where $\lambda_\circ$ was defined by \eqref{pee}.
 
  As the non-vanishing analytic family $\{\Phi_{_{\lambda}}\}_{_{|\lambda| < 1}}\, $ we choose
 \begin{equation}\label{fii}
 \Phi_{_\lambda}(z)= F^\lambda_z(z) (1+\tau_\lambda(z)).
 \end{equation}
  Explicitly,
  \begin{equation*}
  \Phi_{_\lambda}(z) \,=\, \frac{(1+\lambda)\,( 1+  |\mu(z)|\,)\; F^\lambda_z(z)}{(1+\lambda)\, (1\,+\, |\mu(z)|\,)\;  - \; p \,\lambda \;|\mu(z)|}\, \;\neq\,0\,, \quad \textnormal{for all}\;\; z\in \mathbb D
  \end{equation*}\\
  Furthermore, since $\;\; F^{\lambda_\circ} \equiv f\,, \;\; F^{\lambda_\circ}_z \equiv f_z\,, $ 
 \begin{equation}\label{mapping}
 |\Phi _{\lambda_\circ}(z)| =\;\left( 1+ |\mu(z)|\,\right) |f_z| =  \, |f_z| \,+\, |f_{\bar z}| \;=\; |Df|
  \end{equation}

 We shall then apply the Interpolation Lemma \ref{interpolation} in the  measure space $\mathscr M(\mathbb D\,, \sigma)\,$ over the unit disk, where
  $$
    \textnormal d \sigma(z)  = \;\frac{1}{\pi} \,\Big( 1 \,-\, \frac{p\;|\mu(z)|}{1 + |\mu(z)|} \;\Big)\;\textnormal dz
  $$
  We start with  the centerpoint $\lambda =0$, where  the  Beltrami equation reduces to the complex Cauchy-Riemann system $ \,F_{\bar z} \equiv 0\,$ with principal solution   the identity map. Hence $\,F^0_z \equiv 1\,, \;\Phi_0 (z) \equiv 1\,$ and $\,M_0 = \| \Phi_0\|_\infty=1\,$.

The estimate $\, M_1 = \sup_{\lambda \in \mathbb D}\|�\Phi_\lambda \|_2  \leqslant 1\,$ requires just a bit more work.
 First, in view of Lemma \ref{lem:areaineq},
  $$
    \int _\mathbb D J(z, F^\lambda) \,\textnormal d z  \;\leqslant \; \pi
  $$
  with equality  if and only if  $\,F^\lambda(z) \equiv z\,$ outside the unit disk.
Here we find from \eqref{eq:tau} that
\[
 J(z, F^\lambda) = |F^\lambda_z(z)|^2(1-|\mu_\lambda(z)|^2)=|\Phi_\lambda(z)|^2 \left( 1-2 \re 
\frac{\tau_\lambda(z)}{1+\tau_\lambda(z)} \right) = \]
\[
|\Phi_\lambda(z)|^2 \left( 1-p \frac{|\mu(z)|}{1+|\mu(z)|} \re  \frac{2 \lambda}{1+\lambda} \right) \geqslant 
|\Phi_\lambda(z)|^2 \left(1-p \frac{|\mu(z)|}{1+|\mu(z)|} \right)
\]\\
\noindent Hence
$$
|\Phi_{_\lambda} (z) |^2 \; \textnormal d \sigma(z)  \leqslant  \;\frac{1}{\pi}\,J(z, F^\lambda)\,\textnormal d z
$$
and, therefore,
$$
 M_1 \;=\;\sup_{|\lambda| < 1}\; \int_\mathbb D  |\Phi_{_\lambda} (z) |^2 \; \textnormal d \sigma(z) \;\; \leqslant \;\sup_{|\lambda| < 1}\; \frac{1}{\pi} \int  _\mathbb D  J(z, F^\lambda) \,\textnormal dz \;\leqslant\; 1
$$
\vskip7pt
We are now ready to interpolate.  For every $0 \leqslant r <1\,$, in view of the interpolation lemma, we have
$$
  M_r = \;\sup_{|\lambda| = r } \| \Phi_{_\lambda} \| _{\frac{1+r}{r}} \leqslant M_0 ^{\frac{1-r}{1+r}} \, M_1 ^{\frac{2r}{1+r}}  \;\leqslant 1.
$$
It remains to substitute $\,r = \frac{1}{p-1} = \lambda_\circ$\,. The desired inequality is now immediate,
\[
\int_{\mathbb D} \; \left( 1-\frac{p\, |\mu(z)|}{1+\,|\mu(z)|} \right)\; \big | Df(z)\big|^{\,p }\;\textnormal d z
\; = \pi\; \int_{\mathbb D} \;  \big| \Phi_{\lambda_\circ}(z) \big|^{\frac{1+r}{r}}\;\textnormal d \sigma(z) \;
\leqslant \;\pi  
\] \end{proof}

\begin{proof}[Proof of Theorems \ref{MainTh3} and \ref{MainTh2}] To infer Theorem \ref{MainTh2} we extend $\,f$ as the identity outside $\,\Omega$. Since $\Omega$ is bounded, $\int_\Omega |Df|^2 \leqslant K \int_\Omega J(x,f) < \infty$ so that $f\in \mathscr W^{1,2}(\Omega)$. But then one easily verifies that the extended function $f$    defines an element of  $\mathscr W^{1,2}_{loc}(\C)$, and accordingly, a $K$-quasiconformal map of  the entire plane.

Consider then a disk $\, D_R \supset \Omega$. By re-scaling, if necessary, Inequality (\ref{Mainineq}) applies to $\,D_R$ in place of the unit disk and with $\,|D_R|$ in place of $\,\pi$. This yields $\mathscr B^p_{D_R}[f] \leqslant \mathscr B^p_{D_R}[Id] $. On the other hand $\mathscr B^p_{D_R\setminus \Omega}[f] \;=\; \mathscr B^p_{D_R\setminus \Omega}[Id] \,$, by trivial means. Hence Theorem \ref{MainTh2} follows. 

Theorem \ref{MainTh3}, in turn, is a direct consequence of Theorem \ref{MainTh2}. The pointwise condition $\textbf{B}_p\,(Df)\geqslant 0$ for  $f = z + h(z)$, together with the boundary condition $h \in {\mathscr C}^\infty_0(\Omega)$, ensures that $f$ represents  a (smooth) $K$-quasiconformal homeomorphism of $\Omega$ having identity boundary values. Here $p$ and $K$ are related by $p=2K/(K-1)$ and we may apply Theorem \ref{MainTh2} at the borderline exponent.
\end{proof}

\begin{remark} \label{rem:nonvanishing}
Above we employed a standard approximation by smooth homeomorphisms in order to easily ensure that our analytic family of quasiconformal maps has non-vanishing Jacobian everywhere. It might be of interest to observe that actually all analytic families of  $\partial_z$-derivatives of  quasiconformal maps are non-vanishing, i.e.~non-vanishing outside a common set of measure zero. This can be deduced via a standard reduction from the following result:
\end{remark}
\begin{lemma} 
Assume that the measurable dilatation $\mu_\lambda$ depends analytically on the parameter $\lambda\in U$, where $U\subset\C$ is a domain. Suppose also
that $|\mu_\lambda| \leqslant a \,\chi_{\Omega}$ for some $0 \leqslant a < 1$,  for all $\lambda\in U$. 
 
 Let   $f=f(\lambda ,z)$ be  the principal solution of the Beltrami equation $f_{\bar z}=\mu_\lambda f_z. $ Then $\lambda\to f_z$ yields  a non-vanishing analytic family.
\end{lemma}
\begin{proof} 
By localization we may assume that $U=\DD$. If the $\mu_\lambda$ are  smooth functions of the $z$-variable, the non-vanishing property is due to Lemma \ref{aregu}.  
For a  general  non-smooth  $\mu_\lambda$ we introduce the mollification $\mu_{\lambda , n}:=\phi_{1/n}*\mu_{\lambda}$, where  $\phi_\varepsilon$ is a standard approximation of identity. Then for any compact subset  $K\subset \DD$ 
\begin{equation}\label{raja}
\lim_{n\to\infty}\; \sup_{\lambda\in K}\| \mu_{\lambda,n} -\mu_{\lambda}\|_{\mathscr L^p(\Omega)}=0\qquad {\rm for \; any}\;\; p>1.
\end{equation}
Denoting by $f_n=f_n(\lambda, z)$ the principal solution corresponding to $\mu_{\lambda ,n}$ and by choosing $p$ large enough,  we may utilize \cite[Lemma 5.3.1]{AIMb}, 
to see that $\| (f_n)_z - f_z \|_2 \to 0$, uniformly in $\lambda \in K$.
 
Next,    developing  the $\mathscr{L}^2(\Omega)$-valued analytic functions $ f_z$ and $ (f_n)_z$  as a power series in $\lambda$, we see from Cauchy's formula that the Taylor coefficients of $ (f_n)_z$ converge in $\mathscr{L}^2(\Omega)$ to those of 
$ f_z$. Thus,   moving to a subsequence if needed, we have outside a set $E \subset \Omega$ of measure zero,
\begin{equation}\label{pointwise}
 (f_n)_z(\lambda,z )\to f_z(\lambda ,z), \qquad  {\rm local \; uniformly \; in }\;\;\lambda\in\DD. \qquad
\end{equation}

Now, according to Hurwitz's theorem, the limit of a  sequence of non-vanishing analytic functions converging locally uniformly is either everywhere non-zero, or identically zero. Choosing a parameter $\lambda_0$, the derivative 
$f_z(\lambda_0, \cdot) $ can vanish only on a set $E_0$ of measure zero \cite[Corollary 3.7.6]{AIMb}. Thus  $f_z(\lambda,z) \neq 0$ for every $\lambda \in \DD$ and every $z \in \Omega\setminus (E \cup E_0)$.
\end{proof}

The notions of quasiconvexity and rank-one convexity extend to
functions defined on an open subset  $\openset\subset\mathbb R^{2\times 2}$.  Rank-one convexity now demands that
for any $A\in\openset$ and rank-one matrix $X\in\mathbb R^{2\times 2}$ 
the map $t\mapsto E(A+tX)$ is convex in a neighbourhood of zero. The condition for being rank-one concave or null Lagrangian is modified analogously. Similarly, in the definition of quasiconvexity (\ref{columbo}) one simply restricts to linear maps $A$ and perturbations $f\in A+\mathscr C^\infty_0(\Omega)$ such that
$Df(z)\in\openset$ for all $z\in\Omega.$

With the above generalized definitions in mind we  next explore a dual formulation to Theorem \ref{MainTh2}, i.e. we
will pass via the inverse map from expansion estimates to compression
estimates. 
In doing so, we shall restrict to the space $\openset=\mathbb R^{2\times 2}_+$ consisting of matrices with positive
Jacobian determinant.
The  inverse functional takes the form
\[ \mathbf{\hat{E}}(A) := \mathbf{E}(A^{-1}) \cdot \det A, \quad \det A >0, \] 
and it preserves rank-one convexity, quasiconvexity as well as polyconvexity
\cite[p.~211]{ball77}. One should also note that for  $\phi \in \mathscr C^\infty_0(\Omega)$ the condition $Df(z)=A+D\phi (z)\in \mathbb R^{2\times 2}_+$ for all $z\in\Omega$ automatically implies that $f$ is a diffeomorphism $f:\Omega\to A(\Omega)$. This enables one to switch to inverse map if needed.
Our main  observation here is that the procedure of taking inverse functionals naturally leads to a full
one-parameter family of Burkholder functionals.   

Let us recall the standard case \eqref{burk} where one now includes all exponents
$p \geqslant 1$,
\[
\textbf{B}_p\,(A) = \,
\Big(\,\frac{p}{2}  \, \det A  \,+ \;\left(1-\,\frac{p}{2} \right)\,|A|
^2\,\Big)\cdot |A|^{p-2} \,, \quad p \geqslant 1
\]
We set for $ p \leqslant 1$,
\[ \textbf{B}_p\,(A) :=  \Big(\,\frac{p}{2}  \, |A|^2 \,+ \;\left(1-\,
\frac{p}{2} \right)\,\det A \,\Big)\cdot |A|^{-p} \cdot\, (\det
A)^{\,p-1} \,, \quad  \det A > 0.
\]
Then  the functionals  $\textbf{B}_p$ are homogeneous of degree $p$,  depend
continuously on $p$, and we have for the inverse transformation
$\mathbf{\hat{B}}_p = \mathbf{B}_q$, where $p+q=2$. These facts justify
the definitions.
Note that, in particular, we recover the null-Lagrangian cases $p=0$
and $p=2$ as smooth phase transition points between the rank-one convex
and rank-one concave regimes:

\begin{proposition} In the space of matrices $A \in \mathbb R^{2\times
2}_+$
\begin{equation}
A \mapsto \textbf{B}_p (A) \;\;\; \textrm{is} \;\;\;\left
\{\begin{array}{ll}
\textrm{rank-one convex} & \textrm{if $ 0\leqslant p \leqslant 2$} \\
\textrm{null-Lagrangian} &  \textrm {if $ p = 0 $\; or\; $ p= 2 $\;\;
\;}\\
\textrm{rank-one concave} & \textrm {if $ p\leqslant 0\;\; \textrm{or}\;
\;\, p \geqslant 2$. }
\end{array} \right.
\end{equation}
\end{proposition}
The standard case of $p \geqslant 1$ above goes back to Burkholder
\cite{Bu1} and the rest follows from applying the inverse
transformation.
In the complex notation we have, for any $p \in \mathbb{R}$,
$$
 \mathscr B_\Omega ^{\,p}\,[f] \;= \int_{\Omega} \textbf{B}_p (Df) =
\int_{\Omega} \Big(\,|f_{z}|\;\mp\;(p-1)\, |f_{\bar{z}}| \;\Big)\cdot
\Big(\;|f_z|\;\pm\;|f_{\bar{z}}|\;\Big)^{p-1}\, \textrm d z \;, 
$$
where $\;\; \pm\,$ stands for the sign of $\;(p-1)$.

Another application of the inverse transformation, this time to Theorem
\ref{MainTh2} leads to
\begin{theorem}
Let $\,f : \,\Omega{\longrightarrow}\,\Omega\,$ be a $\,K
$-quasiconformal map of a bounded open set $\,\Omega \subset \mathbb C\,
$  onto itself, extending continuously up to the boundary, where it
coincides with the identity map $\,Id(z) \equiv z$.  Then
\begin{equation}
\mathscr B_\Omega ^{\,p}\,[f] \;\,\leqslant\; \mathscr B_\Omega
^{\,p}\,[Id\,]
\quad \textnormal{\textit{for all }} \quad\, -\frac{2}{K-1} \leqslant p
\leqslant 0\,.
\end{equation}
Further, the equality occurs  for a class of (compressing) piecewise
radial mappings discussed in Section \ref{se:radial}.
\end{theorem}

\section{Sharp $\mathscr L\, \textnormal{log} \mathscr L\,$, $\mathscr L^p$ and Exponential 
integrability}\label{se:limitcases}

The sharp integral inequalities provided by  Theorems \ref{MainTh} and \ref{MainTh2}  give us a number of interesting  consequences. We start with the following optimal form of the Sobolev regularity of $K$-quasiconformal mappings.

\begin{corollary} \label{best}
Suppose $\Omega \subset \C$ is any bounded domain and $f:\Omega \to \Omega$ is a $K-$quasiconformal mapping, continuous up to $\partial \Omega$, with $f(z) = z$ for $z \in \partial \Omega$. Then  \begin{equation}\label{Lp1}
\;\;\frac{1}{|\Omega|}\int_{\Omega}  \big | Df(z)\big|^{\,p }\;\textnormal d z
\;\leqslant \; \frac{2K}{ 2K \;-\; p\,(K-1)}\;,\quad for\;\; 2 \leqslant p < \frac{2K}{K-1}\,
\end{equation}\\
The estimate holds as an equality for $f(z) = z|z|^{1/K - 1}$, $z \in \DD$, as well for a family of more complicated maps described in Section \ref{LlogLeq}.
\end{corollary}
\begin{proof} Inequality (\ref{Lp1}) is straightforward consequence of Theorem \ref{MainTh2}, since for  $p < 2K/(K-1)$ we have pointwise $\,{\bf B}_p\bigl( Df(x)\bigr)  \geqslant |Df(x)|^p \;\frac{2K - \;p\,(K-1)}{2K}\,$.
\end{proof}

We next introduce yet another rank one-concave variational integral, simply by differentiating  $\mathscr B^{\,p}_\Omega \,[f]\,$ at $\,p=2\,$,
\begin{equation}\label{Buus}
\mathscr F_\Omega\,[f] := \lim_{p \searrow 2} \frac{\mathscr B^{\,p}_\Omega \,[f]\; -\;\mathscr B^{\,2}_\Omega \,[f]\,}{p \,-\,2}\;=
\end{equation}
$$\; \frac{1}{2} \int _\Omega \Big [\,\left(1 + \log |Df(z)|^2 \right )\, J(z,f)\,\textnormal d z  \;\;-\;\; |Df(z)|^2  \,\textnormal d z \,\Big ]
$$
The nonlinear differential expression
$
  J(z,f) \,\log \,|Df(z)| ^2 \,,$ for mappings with nonnegative Jacobian, is well known to be locally integrable, see \cite{GI} for the following qualitative local estimate on concentric balls $\,B\subset 2B \subset \Omega\,$,
\begin{equation}\label{LlogL}
 \Xint- _B J(z,f) \;\log \Big( e + \frac{\;|Df(z)|^2}{\Xint-_{_B }|Df|^2}  \Big )\;\textnormal d z \;\leqslant \; C\,\Xint-_{2B} |Df(z)|^2 \textnormal d z
\end{equation}
see also Theorem 8.6.1 in \cite{IMb}.  
However, for global estimates one must impose suitable boundary conditions on $\,f$. For example, global $\,\mathscr L\log \mathscr L (\Omega) \,$ estimates follow from (\ref{LlogL}) if $\,f\,$ extends beyond the boundary of $\,\Omega\,$ with finite Dirichlet energy and nonnegative Jacobian determinant. This is the case, in particular, when $\,f(z) - z \in  \mathscr W^{1,2}_{0}(\Omega)\,$.   

 Let us denote the class of homeomorphisms $\, f \in \mathscr W^{1,2}_{\textnormal{loc}}(\mathbb C)\,$ which coincide with the identity map outside a compact set by $\,\mathscr W^{1,2}_{\textnormal{id}} (\mathbb C) \,$.  It is useful to observe (see e.g. \cite[Thm. 20.1.6]{AIMb}) that such a map is automatically uniformly continuous. 
Further, the $\,\mathscr C^\infty$-smooth diffeomorphisms in $\,\mathscr W^{1,2}_{\textnormal{id}} (\mathbb C) \,$ are dense. Precisely, one has
\begin{lemma}[Approximation Lemma, {\cite[Thm. 1.1]{IKO}}
 ]\label{approximation1}
  Given any homeomorphism $\,f \in \mathscr W^{1,2}_{\textnormal{id}} (\mathbb C) \,$, one can find  $\,\mathscr C^\infty$-smooth diffeomorphisms $\,f^\ell \in \mathscr W^{1,2}_{\textnormal{id}} (\mathbb C) \,,\,\ell\geqslant 1\,$,  such that
  $$\|f^\ell\,-f\|_\infty + \,\;\|D(f^\ell\,-f\,)\|_{\mathscr L^2(\mathbb C)}\;\rightarrow 0, \quad  \mbox{ as } \;\ell\rightarrow\infty. $$
Passing to a subsequence if necessary,  we may ensure that $\;Df^\ell \,\longrightarrow Df\,$ almost everywhere.
\end{lemma}

With estimates for the Burkholder integrals we now arrive at  sharp global $\,\mathscr L\log \mathscr L (\Omega) \,$ bounds.
\begin{proof}[Proof of Corollary \ref{LlogL Th}]
Upon the extension  as identity outside $\,\Omega\,$, $\,f \in  \mathscr W^{1,2}_{\textnormal{id}} (\mathbb C) \,$. We use the sequence $\{\,f^\ell\,\}\,$ in the approximation Lemma  \ref{approximation1}, and view each $\,f^\ell\,$ as a principal solution to its own  Beltrami equation
$$
  f^\ell_{\bar z} \;= \;\mu_\ell(z) \,f^\ell_z \;\;, \quad\;\; |\mu_\ell(z) | \leqslant k_\ell <1\;, \;\;\;\mu_\ell(z) = 0\;, \;\; \textnormal{for}\;\; |z| \geqslant R.
$$
where $\,R\,$ is chosen, and temporarily fixed, large enough so that $\,\Omega \subset D_R = \{z\;: |z| < R\}\,$. It is legitimate to apply  Theorem \ref{MainTh2} for each of the maps $\,f^\ell:  \,D_R \to \,D_R\,$,
$$
  \mathscr B^p_{D_R}\,[f^\ell] \leqslant \;|D_R| \;=\;\mathscr B^2_{D_R}\,[f^\ell] \;,\quad\;\; \textnormal{whenever}\;\; 2\leqslant p \leqslant 1 + k_\ell^{-1}
$$
Letting $\,p \searrow 2\,$ we obtain 
\begin{equation} \int _{D_R} \,\left(1\, + \,\log |Df^\ell(z)|^2 \,\right )\, J(z,f^\ell)\,\textnormal d z  \;\leqslant \;\int _{D_R}  |Df^\ell(z)|^2  \,\textnormal d z
\end{equation}
Convergence theorems in the theory of integrals let us pass to the limit when $\,\ell\rightarrow \infty\,$, as follows
$$
\int _{D_R} \,J(z,f)\,\left[1\, + \,\log |Df(z)|^2 \,\right ]\,\textnormal d z  = $$
$$\int _{D_R} J(z,f)\left[1\, + \,\log(1+ |Df|^2) \,\right ]\,\textnormal d z\;
-\int _{D_R} J(z,f)\left[\,\log ( 1 + |Df|^{-2})  \,\right ]\,\textnormal d z $$

$$\leqslant \;\;\liminf_{\ell\rightarrow \infty} \int _{D_R} \,J(z,f^\ell)\,\left[1\, + \,\log(1+ |Df^\ell|^2) \,\right ]\,\textnormal d z $$
$$-\lim_{\ell\rightarrow \infty}\;\; \int _{D_R} \,J(z,f^\ell)\,\log \left( 1 + |Df^\ell|^{-2}  \,\right )\,\textnormal d z  $$
 Here the $(\liminf)$-term is justified by Fatou's theorem while the $(\lim)$-term by the Lebesgue dominated convergence,  where we observe that the integrand is dominated point-wise by $ J(z,f^\ell)\,| Df^\ell| ^{-2}\,\leqslant 1$. The lines of computation continue as follows
 $$
  =\;\;\liminf_{\ell\rightarrow \infty} \int _{D_R} \,J(z,f^\ell)\,\left[1\, + \,\log|Df^\ell|^2 \,\right ]\,\textnormal d z \leqslant
 $$

 $$
 \liminf_{\ell\rightarrow \infty} \int _{D_R}  |Df^\ell|^2 \,\,\textnormal d z  = \int _{D_R} \, |Df(z)|^2 \,\textnormal d z
 $$
Finally, we observe that
$$
\int _{D_R \setminus \Omega} \,J(z,f)\,\left[1\, + \,\log |Df(z)|^2 \,\right ]\,\textnormal d z  = \int _{D_R\setminus\Omega} \, |Df(z)|^2 \,\textnormal d z\;,$$
which combined with the previous estimate yields \eqref{LlogL1},
as desired.
\end{proof}

 The above energy functional can be further cultivated by applying  the inverse map, as described at the end of Section \ref{se:proof}.  A search of minimal regularity in  the corresponding integral estimates leads us to  mappings of integrable distortion.
 
 \begin{corollary}\label{loginv} Let  $\Omega\subset\R^2$ be a bounded domain, and suppose  $h \in {\mathscr W}^{1,1}_{loc}(\Omega)$  is a homeomorphism of $\overline \Omega$ such that $h(z) = z$ for $z \in \partial \Omega$.
 Assume $h$ satisfies the distortion inequality
\begin{equation}\label{distortion2}
|D h(z)|^2\leqslant K(z)J(z,h),\qquad {\rm a.e\;\; in} \;\; \Omega,\nonumber
\end{equation}
where $1\leqslant K(z)<\infty$ almost everywhere in $\Omega.$ The smallest such function, denoted by $ K(z,h)$, is assumed to be integrable. 
 Then
\begin{equation}\label{loginvar} 2 \int_\Omega \log |Dh| - \log J(z,h) \; \leqslant  \;  \int_\Omega  K(z,h) -  J(z,h)
 \end{equation}
In particular,  $\log J(z,h)$ is integrable. Again there is a wealth of functions, to be described in Section 5, 
 satisfying  \eqref{loginvar} as an identity.
 \end{corollary}
 Note that in fact
$$
K(z,h)=\left\{
\begin{array}{ll}
\frac{|Dh(z)|^2}{J(z,h)}, & {\rm if}\;\;  J(z,h) > 0,\\
1, & {\rm otherwise}.\;\;  \\
\end{array}\right.
$$
Thus according to the corollary, the functional ${\mathscr H} (A)$ in \eqref{functional} is quasiconvex at $A= Id$, in its entire natural domain of definition  $\openset=\mathbb R^{2\times 2}_+$.
 
\begin{proof}[Proof of Corollary \ref{loginv}]
It was observed in \cite{AIMO} that if $h:\Omega\to h(\Omega )$ is a homeomorphism of Sobolev class $\mathscr W^{1,2}_{loc} (\Omega )$ with integrable distortion $K(z,h)\in \mathscr L^1 (\Omega ),$ then its inverse map $f= h^{-1}: h(\Omega )\overset{{\rm onto}}{\to}\Omega $ belongs to $\mathscr W^{1,2} (h(\Omega ))$ and one has 
\begin{equation}\label{aimo}
\int_{h(\Omega )} |Df|^2=\int_\Omega K(z,h).
\end{equation}
Enhancement of this identity for mappings $h\in\mathscr W^{1,1}_{loc}(\Omega )$ is given in \cite{HKO}.

In our situation $f(z)-z=h^{-1}(z)-z$ is continuous in $\overline{\Omega}$, vanishes on $\partial\Omega$,
and belongs to $\mathscr W^{1,2}(\Omega )$. It is well-known that this implies  $f(z)\in z+\mathscr W^{1,2}_0(\Omega ).$ Thus it is legitimate to apply  Corollary \ref{LlogL Th} to deduce
\begin{equation}\label{jj}
\int_{\Omega } 1+ 2\log \Bigl(\frac{|Dh|}{J(z,h)}\Bigr)\leqslant \int_\Omega K(z,h).
\end{equation}
Here we used the identity $|A^{-1}| = |A|/\det(A)$ and a change of variable, made legitimate by the fact  that $f$ is a homeomorphism of Sobolev class $\mathscr W^{1,2} (\Omega ).$ The claim obviously follows from (\ref{jj}), and Hadamard's inequality  yields the logarithmic integrability, 
$$- \int_\Omega  \log J(z,h) \; \leqslant 2 \int_\Omega \log |Dh| - \log J(z,h),
$$
with equality here for the indentity mapping.
\end{proof}

In particular,  the Jacobian of a map with integrable distortion cannot approach zero too rapidly. This result is in fact known \cite{KO}, the novelty in (\ref{loginvar}) is the sharpness.

We next turn to  the exponential integrability results, which will follow from  Theorem \ref{MainTh} at the limit $p\to\infty .$ 
\begin{proof}[Proof of Corollary \ref{expint}] Let us assume we are given a function $\mu$,  supported in $\DD\,$ with $ |\mu (z)|\leqslant 1$ for all $z\in\DD.$  We then consider the principal solution $f$ of the
Beltrami equation
$
f_{\overline z}\; =\;\varepsilon \mu f_z
$
and apply Theorem  \ref{MainTh} with $k=\varepsilon$ and $p=1+1/\varepsilon$ to obtain
\begin{equation}\label{epsi}
\int_\DD \left(\frac{1-|\mu (z)|}{1+\varepsilon |\mu (z)|}\right)\big| \, Df(z)\, \big|^{1+1/\varepsilon}\, dz\leqslant \pi .
\end{equation}
By applying the Cauchy-Schwarz inequality and the ${\mathscr L}^2$-isometric property of ${\mathbf S}$, wee see that for
almost every $z\in \DD$, developing \eqref{SBel}-\eqref{56} to  a  Neumann
series represents  $f_z $  as a power series in  $\varepsilon$, with  convergence radius $ \geqslant 1$. Hence
$$
f_z=1+\varepsilon {\mathbf S}\mu + O(\varepsilon ^2)\qquad{\rm for \; a.e.}\;\; z\in\DD.
$$
We may use this  to compute pointwise 
\begin{eqnarray}
(1+1/\varepsilon )\log |Df|&= &(1+1/\varepsilon )\big( \log (1+\varepsilon |\mu |)+\log |f_z|\big)\nonumber\\
&=& |\mu |+\re {\mathbf S}\mu +O(\varepsilon ).\nonumber
\end{eqnarray}
Hence  $\big| \;Df\; \big|^{1+1/\varepsilon} = \exp (|\mu |+\re{\mathbf S}\mu ) + O(\varepsilon )$ and the desired result follows at the limit $\varepsilon\to 0$ by an application of Fatou's lemma on (\ref{epsi}).
\end{proof}

\section{Piecewise Radial Mappings}\label{se:radial}

\subsection{Examples of optimality in Theorems  \ref{MainTh3}, \ref{MainTh2} and  \ref{MainTh}}

Our exposition here is slightly condensed since the basic principle behind these examples can be found already in  the paper \cite{BM} of A. Baernstein and S. Montgomery-Smith, or in the paper  \cite{Iw1} of the second author.
Let us start by  describing the building block of the maps that yields equality in our main result. For any $0\leqslant r<R$ consider the radial map
\begin{equation}
\label{rad}
g(z) = \rho(\,|z |\,) \;\frac{z  }{|z|}
\end{equation}
defined in the disc $\{ |z|\leqslant R\}.$ We assume that $\rho:[0,R]\to [0,R]$ is  absolutely continuous and strictly increasing with $\rho(0)=0$, and that  $\rho$ is linear on $[0,r]$. We first restrict ourselves to the situation $p\geqslant 1$, and then   need the following \emph{expanding assumption}
\begin{equation}\label{expanding}
  \frac{\rho(t)}{t}\; \geqslant\; \dot{\rho}(t)\; \geqslant 0\;,\qquad t\in (r, R)
\end{equation}
together with the normalization $\rho (R)=R.$  Hence on the boundary
the map $g$ coincides with the identity map, and if needed we may extend $g$ to the exterior $\{ |z|\geqslant R\}$ by setting $g(z)=z$ for these values. 

The differential of $g$ exhibits the following rank-one connections
\begin{equation}
 \label{eq:diff}
Dg(z)= \frac{\rho(|z|)}{|z|} Id + \left( \dot{\rho}(|z|) - \frac{\rho(|z|)}{|z|} \right) \frac{z \otimes z}{|z|^2}.
\end{equation}
It is known, see \cite[Proposition 3.4]{Ball90} that concavity along the indicated rank-one lines already secures the quasiconcavity
condition for the radial map $g$.
In our situation the assumption (\ref{expanding}) indeed ensures that the
Burkholder integrals become linear on the rank-one segments displayed in \eqref{eq:diff}, which implies
\begin{equation}
\label{eq:radialquasiconcavity}
 \mathscr B_{B(0,R)} ^{\,p}\,[g] = \mathscr B_{B(0,R)} ^{\,p}\,[Id].
\end{equation}
 Actually, a direct computation (see  \cite{BM,Iw1} for details)
using the formulas 
$
  g_z(z)   = \frac{1}{2} \Big(\, \dot{\rho}(|z|)\;+{\rho(|z|)}/{|z|}\;\Big )$ and $g_{\bar{z}}(z)   = \frac{1}{2} \Big(\, \dot{\rho}(|z|)-{\rho(|z|)}/{|z|}\;\Big ){z}/{\bar{z}}
  $
yields
\begin{eqnarray}\label{Burkholder4}
 \mathscr B_{B(0,R)} ^{\,p}\,[g] &=&\pi\int_0^R \left( \frac{[\,\rho(t)\,]^ p}{\, t^{p-2}}\nonumber
     \right)'\;\textrm {d} t  \;=\;  \pi\;\frac{[\,\rho(R)\,]^p}{\, R^{\,p-2}} \;-\; \lim_{t\to 0^+} \pi\;\frac{[\,\rho(t)\,]^p}{\, t^{\,p-2}}\\
     &=& \pi R^2=|B(0,R) | =  \mathscr B_{B(0,R)} ^{\,p}\,[Id] .
     \nonumber
     \end{eqnarray}
The above computation indicates that if $\,r = 0\,$,  we must  in addition require  
\begin{equation}\label{aa1}
\,\rho(t) = o( t^{1-\frac{2}{p}})\,\qquad  {\rm as}\;\; \; t\to 0\,.
\end{equation}

Assume then that $f_0(z)=az+b$ is a (complex) linear map defined in a bounded domain $\Omega\subset\C.$ Given $0\leqslant r<R$ and a   ball $B(z_0,R)\subset \Omega$ together with the increasing homeomorphism $\rho:[0,R]\to [0,R]$ and the radial map $g$ as discussed above, we may modify $f_0$
in $B(z_0,R)$ by defining
$$
f_1(z)=\left\{
\begin{array}{ll}
f_0(z) & {\rm if}\;\;  z\not\in \overline{B(z_0,R)},\\
ag(z-z_0)+(az_0+b) & {\rm if}\;\;  z \in{B(z_0,R)},\\
\end{array}\right.
$$
By scaling, (\ref{eq:radialquasiconcavity}) shows that we have $ \mathscr B_{B(z_0,R)} ^{\,p}\,[f_1] = \mathscr B_{B(z_0,R)} ^{\,p}\,[f_0],$ and, consequently
\begin{equation}
\label{eq101}
 \mathscr B_{\Omega }^{\,p}\,[f_1] = \mathscr B_{\Omega }^{\,p}\,[f_0].
\end{equation}
In the next step we may deform $f_1$ in a disc that is contained in either  one of the sets  $\Omega\setminus B(z_0,R)$ or  $ B(z_0, r)$, where $f_1$ is linear. Inductively one obtains $f_n$ from $f_{n-1}$ by deforming $f_{n-1}$ accordingly in the domains of linearity. 
   By induction, we see that all such mappings  have the same energy, which is equal to the energy of their linear boundary data $\,az + b$:
\begin{equation}\label{energy}
  \mathscr B^{\,p}_{\Omega}
 \,[ f_{n}]\;=\; \mathscr B^{\,p}_{\Omega} \,[az+b] \; =\; a^p\,|\Omega|
\end{equation}

This iteration process may, but need not, continue indefinitely so as to arrive at e.g.  Cantor type configuration of annuli and a homeomorphism $\,f_\infty :\,\Omega \overset{\textnormal{\tiny{onto}}}{\longrightarrow}\,\Omega^*\,:=a\Omega +b$.    Without going to the formal definition, we loosely refer to reasonable (e.g. converging in $\mathscr W^{1,1}$) such limits $f=\lim_{n\to\infty} \,f_n$  as \textit{piece-wise radial mappings}, see Figure 1.

\begin{figure}
\includegraphics[width=4.5in]{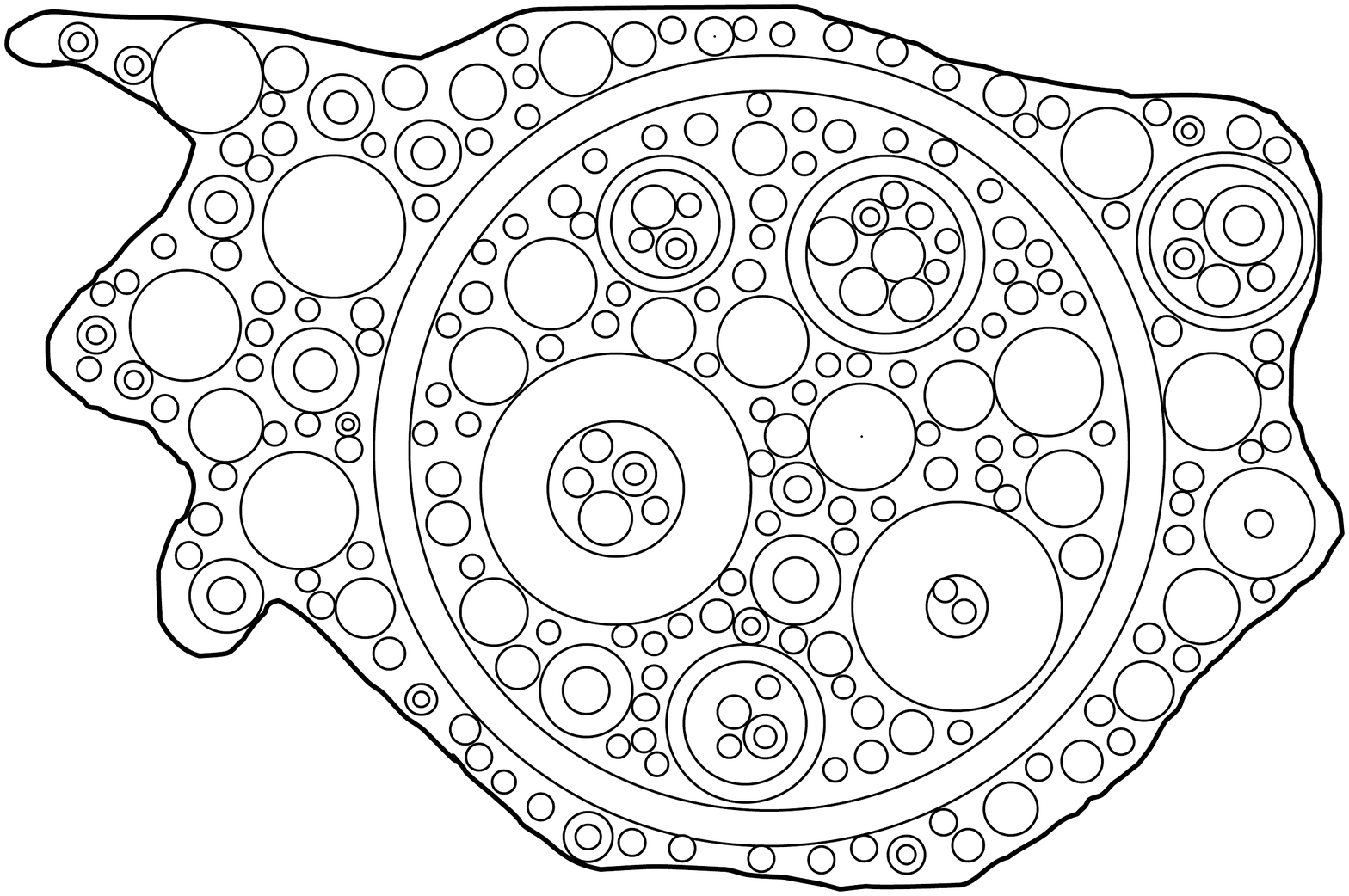}
\caption{Annular Packing}
\end{figure}

\begin{definition} Let $p\geqslant 1$ and let  $\,\Omega \subset \mathbb C\,$ be any nonempty 
 bounded domain.  
The class $\radial^p(\Omega)\,$ consists  of  piece-wise radial mappings 
  $\,f_\infty : \,\Omega \overset{\textnormal{\tiny{onto}}}{\longrightarrow}\, \Omega\,$   whose construction starts with $\; f_0(z) \equiv z\,$, the convergence $f_\infty=\lim_{n\to\infty} f_n$ takes place in $\mathscr W^{1,p}(\Omega )$ and the condition (\ref{expanding}) is in force.
\end{definition}
\noindent The following observation is a direct corollary of (\ref{energy}).
\begin{proposition}\label{Equality}
For any $p \geqslant 1$ and   $\,f\in \radial^p(\Omega$) we have
\begin{equation}\label{aaaaa}
\,\mathscr B_{\Omega} ^{\,p}\,[f] = \mathscr B_{\Omega} ^{\,p}\,[\textnormal{Id}] \,= |\Omega|.
\end{equation}
\end{proposition}
\noindent In interpreting this conclusion one may say that $\,\mathscr B_{\Omega} ^{\,p}\,$ is a null Lagrangian when restricted to   $\,\radial^p(\Omega)\,$.

We get a plethora of fairly complicated maps that produce equality in our main Theorems.
One just needs to consider any $f\in \,\radial^p\,$ that satisfies the additional
condition $\mathbf B_p(Df(x)) \geqslant 0,$ i.e. in the construction one applies maps  $\rho$ that
satisfy
\begin{equation}\label{rho4}
\frac{\rho(t)}{t} \, \geqslant \, \dot{\rho}(t)\, \geqslant \, \Big(1 - \frac{2}{p}\,\Big)\, \frac{\rho (t)}{t}.
\end{equation}
In case of Theorem \ref{MainTh3}, in order to satisfy the smoothness assumption  one of course has to pick the functions $\rho$ in the construction  so that the (possibly limiting) map belongs to $Id + {\mathscr C}^\infty_\circ (\Omega ).$

 In order to incorporate the range  $p\leqslant 1$, one introduces  the \emph{compressing assumption}, the opposite of 
(\ref{expanding}):
 \begin{equation}\label{nonexpanding}
  \frac{\rho(t)}{t}\; \leqslant\; \dot{\rho}(t)\;,\qquad t\in (r, R)
\end{equation}
together with the condition
$  \lim_{t\to 0^+}\rho(t)/t^{1-2/p}=\infty$
in case $r=0$ and $p<0$.
The previous construction yields after finitely many iterations non-trivial examples of maps that satisfy 
(\ref{aaaaa}) in case $p<1.$ Also limiting maps can be considered, but then one has to take care of  the Burkholder functional to remain well-defined.

In view of Conjecture \ref{Morrey's}, the maps in $\,\radial^p(\Omega)\,$ are potential global
 extremals for $\,\mathscr B_{\Omega} ^{\,p}\,$. Indeed, it can  be shown that they are critical points of the associated Euler-Lagrange 
equations. Furthermore, this property to a large extent characterizes Burkholder 
functionals: these functionals are the only (up to scalar multiple) isotropic and homogeneous 
variational integrals with  the $\,\radial^p(\Omega)\,$ as their  critical points. This topic will be discussed in a companion paper. 

Here, we 
content with pointing out the following result, where we employ  the customary  notation $\,\mathscr C^1_{\textnormal {id}}(\Omega)\, = \, \textnormal{id} \,+ \mathscr C^1_\circ(\Omega)\,$, where  $\, \mathscr C^1_\circ(\Omega)\,$ stands for the space  of functions $\mathscr C^1$ -smooth up to the boundary of $\,\Omega\,$ and vanishing on $\,\partial \Omega\,$.

\begin{corollary}\label{locmaxima}
The Burkholder functional $\;\mathscr B^{\,p}_\Omega\; :  \;\mathscr C^1_{\textnormal{id}}(\Omega) \rightarrow \,\mathbb R\,$, $ \,p> 2\,,$ attains its local maximum at every $\,\mathscr C^1$- smooth piece-wise radial map in $\, \radial^p(\Omega)\, $ for which the condition $(\ref{rho4})$ is  further reinforced  to:
\begin{equation}\label{rho3}
\frac{\rho(t)}{t} \, \geqslant \, \dot{\rho}(t)\, \geqslant\, \frac{1}{K} \, \frac{\rho (t)}{t}, \qquad K< \frac{p}{p-2}.
\end{equation}
\end{corollary}
\begin{proof}The lower bound (\ref{rho3}) can be used to verify that
$f$ is $K$-quasiconformal with $K<p/(p-2).$
 Namely, as  $\,f \in\mathscr C^1_{\textnormal{id}}(\Omega)\,$, one checks that $f$ is necessarily conformal at points corresponding to $t=0$ and the derivative is non-vanishing.  Hence the strict inequality for $K$ will not be destroyed by small $\,\mathscr C^1$ -perturbations.  Theorem \ref{MainTh2} applies, and by combining it with Proposition \ref{Equality} the claim is evident. 
\end{proof}

\subsection{Equality in Corollaries \ref{LlogL Th},   \ref{expint},  \ref{best} and \ref{loginv}} \label{LlogLeq} 

If one substitutes in the formula \eqref{Buus} a function for which $\;\mathscr B^{\,p}_\Omega [f] \; = \;\mathscr B^{\,2}_\Omega [f]$, we  acquire
 the equality in the  $\mathscr{L\log L}$-inequality  of  Corollary \ref{LlogL Th}. Especially, by (\ref{aaaaa}) we obtain
\begin{lemma} Let $\Omega$ be a bounded domain in the plane.
If $f$ belongs to class $\bigcup_{p>2}\radial^p(\Omega)$, then there is equality in \eqref{LlogL1}  in Corollary \ref{LlogL Th}.
\end{lemma}
\noindent Actually, one checks that in the construction condition (\ref{aa1}) can be replaced by the analogue
$\rho (t)=o\big(\log (1/t)^{-1}\big).$ Concerning Corollary \ref{loginv}, an alternative route to its  integral estimates comes by taking the derivative  $\partial_p \mathscr B_\Omega ^{\,p}\,[f]$ at $p=0$.
Thus elements of $\bigcup_{p<0}\radial^p(\Omega)$ give the identity at \eqref{loginvar}.

We next turn our attention to  Corollary \ref{expint}. It turns out that there, as well, one has a very extensive class of functions $\mu$ of radial type that yield an equality in the estimates. These functions can be viewed as infinitesimal generators of the expanding class of radial mappings defined above. 
 
 \begin{lemma}\label{exponentialexample}   Let  $\alpha \colon (0,1)\to [0,1]$ be measurable and with the property
 $$\int_0^1 \frac{1-\alpha(t)}{t}  dt= \infty . 
 $$
 Set $$\mu(z) = - \frac{z}{\bar z} \alpha(|z|) \quad \mbox{ for  } |z| < 1,   \qquad \mu(z) = 0\quad \mbox{ for } |z| \geqslant 1.$$ Then
 there is equality in Corollary \ref{expint}, i.e.
 $$  \int_\DD (1-|\mu(z)|) \;e^{\, |\mu(z)| \,}\,  \bigl| \exp({\mathbf{S}\mu(z)}) \bigr| \;dz = \pi
$$
\end{lemma}

\begin{proof} Let $\phi(z) = 2 z \int_{|z|}^1 \frac{\alpha(t)}{t} \, dt$ for $|z| < 1$ and set $\phi(z) = 0 $ elsewhere. Then we compute that $\phi \in {\mathscr W}^{1,2}(\C)$ with
$$ \phi_{\bar z} \equiv \mu, \qquad \phi_z = {\beur} \mu (z) = 2\int_{|z|}^1 \frac{\alpha(t)}{t} \, dt - \alpha(|z|), \quad |z| < 1.
$$
Thus
\begin{eqnarray*}
&& \hspace{-1.5cm} \int_\DD (1-|\mu(z)|) e^{|\mu(z)| + \re  \,{\beur} \mu (z)} dm \\ && = 
2\pi\int_0^1 \bigl(1-\alpha(t)\bigr) \exp\left[ 2\int_t^1 \frac{\alpha(s)}{s} \, ds   \right] \;t \, dt  = \pi,
\end{eqnarray*}
as we have the identity
$$ \frac{d}{dt} \left( t^2  \exp\left[ 2\int_t^1 \frac{\alpha(s)}{s} \, ds   \right]  \right) = 2t  \bigl(1-\alpha(t)\bigr)  \exp\left[ 2\int_t^1 \frac{\alpha(s)}{s} \, ds   \right]
$$
and our assumption gets rid of the substitution at $t=0$.
\end{proof}
More complicated examples may be obtained by a similar iteration procedure as described above.

Finally, equality in (\ref{Lp1}) obviously implies that necessarily  the distortion function $K(z,f) \equiv K$ in $\Omega$. Hence examples are produced by specific functions $\rho$ in \eqref{rad}, the powers
$$\rho_K(t) \; = \; R^{1-1/K} \;t^{1/K}, \quad r <  t < R,
$$
where $\rho_K(t)$ is linear on $(0,r]$, if $r >0$. For $2 \leqslant p < \frac{2K}{K-1}$ 
let $\,\radial_{K}^p(\Omega)\,$ denote the subclass of $\,\radial^p(\Omega)\,$ consisting of  those piecewise radial mappings where, first, we fill the domain $\Omega$ by discs or annuli up to measure zero, second,  at each construction step  choose $\rho = \rho_K$, and third, choose $r=0$ at any possible subdisk remaining in the limiting packing construction.
This ensures that the limiting function $f$ does not remain linear  in any subdisk,  so that we have $K(z,f) \equiv K$ up to a set of measure zero. Then, as  $|\Omega|<\infty ,$  it is easy to see that convergence $f_\infty=\lim_{n\to\infty} f_n$ takes place in $\mathscr W^{1,p}$ since now $1- p|\mu_n (z)|(1+|\mu_n (z)|)^{-1} \geqslant c_0>0$. Moreover, since there is equality in
Theorem \ref{MainTh} and $K(z,f) \equiv K$, one obtains for any $f\in \,\radial_{K}^p(\Omega)\,$ that
$$
\frac{1}{|\Omega|}\int_{\Omega}  \big | Df(z)\big|^{\,p }\;\textnormal d z
=  \; \frac{2K}{ 2K \;-\; p\,(K-1)}\;.
$$

\subsection{On $\,\mathscr L^p$-bounds for Quasiconformal Mappings in $\mathbb R^n$ }
Let us close with a discussion on higher dimensional analogues of our optimal results in plane.
This leads to new conjectures on the $\,\mathscr L^p$-regularity problem of quasiconformal maps in space. We begin with the following statement in all dimensions.
 \begin{theorem} \label{Bp}The functions $\;\textbf{B}_p^{\,n} : \mathbb R^{n\times n} \rightarrow \mathbb R\,$  defined by
\begin{equation}\label{BpInt}
\textbf{B}_p^{\,n}\,(A) = \,\Big(\; \frac{p}{n}\det A  \; +\; (1-\frac{p}{n})\, \big{|}A\big{|}^n \; \Big)\cdot |A|^{p-n},\qquad  p\geqslant {n}.
\end{equation}
are rank-one concave.
\end{theorem}
For the proof  and more information on this topic see \cite{Iw1}. Of course, one may be tempted to show that $\textbf B_p^{n} \,$ is quasiconcave, but that is beyond our reach in such a generality. 
However, we state here  potential consequences of the conceivable quasiconcavity of the above $n$-dimensional Burkholder functional for quasiconformal maps, i.e. analogues of our optimal results in plane.

\begin{conjecture}\label{sharp}
 Suppose $\,f:\mathbb B \rightarrow \mathbb B\,$ is a  $\,K\,$-quasiconformal mapping of the unit ball onto itself that is equal to the identity on $\,\partial \mathbb B\,$.
Then,
\begin{equation}
{\Xint-}_{\!\!\mathbb B}\left( n - p \; +\;\frac{p}{K(x)} \right)\;\left| D\!f(x)\right| ^p \;\textrm{d} x \leqslant n\;,\;\;\;\; for \;\;n \leqslant p \leqslant  \frac{nK}{K-1}
\end{equation}
Therefore,
\begin{equation}
{\Xint-}_{\!\!\mathbb B}\left| D\!f(x)\right| ^p \;\textrm{d}x  \;\;\leqslant \;\; \frac{nK}{nK - p(K-1)}
\end{equation}
There are also sharp estimates with the critical upper exponent  $\;p = \frac{n K}{ K-1}\,$,
\begin{equation}
{\Xint-}_{\!\!\mathbb B}\left( \frac{1}{K(x)}\; -\;\frac{1}{K} \right)\;\big| D\!f(x)\big| ^{\frac{nK}{K-1}} \;\textrm{d}x \;\;\leqslant \;\; 1 - \frac{1}{K}
\end{equation}
Hence
\begin{equation}
\int_\mathbb E \big|D\!f(x)\big|^{\frac{nK}{K-1}} \;\textrm{d} x  \leqslant |\mathbb B| \,,\;\;\;\;\textrm{where}\;\;\;\;  \mathbb E = \{ x\in \mathbb B\; ;\;\;  K(x) = 1\;\}
\end{equation}
All the above inequalities are sharp; $\,n$-dimensional radial and power mappings produce examples for the equalities.
\end{conjecture}

\bigskip

\emph{Acknowledgements.}
{Astala was supported by  Academy of Finland grant 11134757 and by the EU-network CODY. Iwaniec was supported by the NSF grant DMS-0800416 and Academy of Finland grant 1128331. Prause  was supported by project 1134757 of the Academy of Finland and by the Swiss NSF. Each of the authors was supported by the Academy of Finland CoE in Analysis and Dynamics research, grant 1118634.

Part of the research took place when the first author was visiting UAM and ICMAT at Madrid and MSRI at Berkeley. K.A thanks both institutes for the inspiring atmosphere and warm hospitality.}

\bibliographystyle{amsplain}

\end{document}